\documentclass[12pt,a4paper]{amsart}

\usepackage[left=3.3cm,right=3.3cm,top=4cm,bottom=4cm]{geometry}
\usepackage[utf8]{inputenc}
\usepackage{amsmath}
\usepackage{amsfonts}
\usepackage{amssymb}
\usepackage{amsthm}
\usepackage[all]{xy}
\usepackage{color}
\usepackage{hyperref}
\usepackage{xypic}
\usepackage{mathrsfs}
\usepackage{centernot}
\usepackage{tikz-cd}

\theoremstyle{plain}
\newtheorem{theorem}{Theorem}

\newtheorem{lemma}[theorem]{Lemma}

\newtheorem{proposition}[theorem]{Proposition}

\DeclareMathOperator{\Div}{div}

\DeclareMathOperator{\Gal}{Gal}

\DeclareMathOperator{\Pic}{Pic}

\DeclareMathOperator{\Stab}{Stab}

\newcommand{\del}{\partial}
\newcommand{\delbar}{\overline{\partial}}

\newcommand{\IN}{\mathbb{N}}

\newcommand{\IZ}{\mathbb{Z}}

\newcommand{\IQ}{\mathbb{Q}}

\newcommand{\IQbar}{\overline{\mathbb{Q}}}
\newcommand{\IR}{\mathbb{R}}
\newcommand{\IC}{\mathbb{C}}

\newcommand{\IP}{\mathbb{P}}

\newcommand{\vol}{\mathrm{vol}}

\newcommand{\dprime}{{\prime\prime}}

\newcommand{\pr}{\mathrm{pr}}

\newcommand{\caO}{\mathcal{O}}

\newcommand{\id}{\mathrm{id}}

\newcommand{\hhat}{\widehat{h}}

\newcommand{\Lotil}{\overline{\mathcal{L}}}
\newcommand{\Motil}{\overline{\mathcal{M}}}

\usepackage{bm}

\makeatletter
\newcommand\footnoteref[1]{\protected@xdef\@thefnmark{\ref{#1}}\@footnotemark}
\makeatother

\usepackage{boondox-cal}
\usepackage[new]{old-arrows}

\newcommand{\volh}{\widehat{\mathrm{vol}}}

\begin{document}
	
\author{Lars K\"uhne}
\email{lars.kuehne@ucd.ie}
\address{UCD School of Mathematics and Statistics \\ 
University College Dublin \\ 
Belfield, Dublin 4 \\
Ireland}


\title[Equidistribution in Families of Abelian Varieties]{Equidistribution in Families of Abelian Varieties and Uniformity}

\begin{abstract}

Using equidistribution techniques from Arakelov theory as well as recent results obtained by Dimitrov, Gao, and Habegger, we deduce uniform results on the Manin-Mumford and the Bogomolov conjecture. For each given integer $g \geq 2$, we prove that the number of torsion points lying on a smooth complex algebraic curve of genus $g$ embedded into its Jacobian is uniformly bounded. Complementing recent works of Dimitrov, Gao, and Habegger, we obtain a rather uniform version of the Mordell conjecture as well. In particular, the number of rational points on a smooth algebraic curve defined over a number field can be bounded solely in terms of its genus and the Mordell-Weil rank of its Jacobian.

\end{abstract}

\maketitle	

Throughout this article, $K$ is a number field and $S$ is a smooth, geometrically irreducible algebraic variety over $K$. The variety $S$ serves as the base of a family $\pi: A \rightarrow S$ of abelian varieties (i.e., $A$ is an abelian scheme over $S$). We furthermore assume that we are given an arbitrary immersion $\iota: A \hookrightarrow \IP^N_K$. Note that such an immersion need not exist and it does not if the base $S$ is not quasi-projective.

For a quasi-projective variety $X$ over $K$ and an archimedean place $\nu$ of $K$, we denote by $X_{\IC_\nu}^{\mathrm{an}}$ the complex (analytic) space $\IC_\nu$-analytic space associated with $X_{\IC_\nu}$ (see \cite{Grauert1994} for this notion). For each archimedean place ${\nu \in \Sigma_\infty(K)}$, a closed point $x \in A$ yields a $0$-cycle $\mathbf{O}_\nu(x) = (x \otimes_K \IC_\nu)^{\mathrm{an}}$ on the $\IC_\nu$-analytic space $A_{\IC_\nu}^{\mathrm{an}}$ associated with $A$. Given an irreducible subvariety $X \subseteq A$, we call a sequence $(x_i) \in X^\IN$ of closed points \textit{$X$-generic} if none of its infinite subsequences is contained in a proper algebraic subvariety of $X$. Note that a sequence is $X$-generic if and only if it converges to the generic point of $X$ in the Zariski topology.

Our first aim is to state an analogue of the equidistribution conjecture for abelian varieties that takes its place within the given relative setting $\pi: A \rightarrow S$. For this purpose, we have to introduce a generalization of the Néron-Tate height. We refer to Section \ref{section:heights} for details and only sketch the basic definitions here. The line bundle $\mathcal{O}(1)$ on $\IP^N_{\mathcal{O}_K}$ can be endowed with $\mathscr{C}^\infty$-hermitian metrics at the infinite places $\Sigma_\infty(K)$ (e.g.\ Fubini-Study metrics), so that we obtain a $\mathscr{C}^\infty$-hermitian line bundle $\overline{\mathcal{O}}(1)$ on $\IP^N_{\mathcal{O}_K}$. 
In this way, we obtain an associated Arakelov height $h_{\overline{\mathcal{O}}(1)}$ for subvarieties of $\IP^N_K$. We remark that our height is the top arithmetic intersection number divided by the degree as this definition is well-adapted to Zhang's inequalities \cite[Theorem 5.2]{Zhang1995} and Yuan's bigness theorem \cite{Yuan2008}. Other articles, most prominently \cite{Faltings1991}, define the height as the top arithmetic intersection number.

Throughout this article, we work with a fixed integer $n \geq 2$. For each irreducible subvariety $X \subseteq A$ and every integer $k \geq 0$, we write $X_k$ for $\iota([n^k](X))$ and $\overline{X}_k$ for its Zariski-closure in $\IP^N_K$. We then define
\begin{equation}
\label{equation::nerontateheight}
\hhat(X) = \limsup_{k \rightarrow \infty} \left(\frac{ h_{\overline{\mathcal{O}}(1)}(\overline{X}_k)}{n^{2k}} \right) \in [0,\infty].
\end{equation}
As $[n]$ preserves the fibers of $\pi$, the limit superior in \eqref{equation::nerontateheight} specializes to an ordinary Néron-Tate height if $X$ is completely contained in such a fiber (\cite[Theorem 11.18]{Gubler2003}). In this case, the limit superior can be replaced by a limit. In particular, the limit process is well-behaved for closed points in $A$. 

For each archimedean place $\nu \in \Sigma_\infty(K)$ and every non-empty simply connected open $U \subseteq S^{\mathrm{an}}_{\IC_\nu}$, there exists by \cite[Proposition B.2]{Dimitrov2021a} a real-analytic isomorphism $$a: A_{\IC_\nu}^{\mathrm{an}}|_U \longrightarrow (\IR/\IZ)^{2\dim(A/S)} \times U$$ that restricts to a group homomorphism in the fiber over each point of $U$. The induced map $$b = \pr_1 \circ a: A_{\IC_\nu}^{\mathrm{an}}|_U \longrightarrow (\IR/\IZ)^{2\dim(A/S)}$$ is unique up to post-composition with an automorphism of $(\IR/\IZ)^{2\dim(A/S)}$. For a geometrically irreducible subvariety $X \subseteq A$ with $\pi(X) = S$ and every point $x \in (X^{\mathrm{sm}})_{\IC_\nu}^{\mathrm{an}} \cap \pi^{-1}(U)$, we define $\mathrm{rank}_{\mathrm{Betti}}(X,x)$, the \textit{Betti rank} of $X$ at $x$, as the $\IR$-dimension of
\begin{equation*}
db(T_{\IR,x}X_{\IC_\nu}^{\mathrm{an}}) \subseteq T_{b(x)}(\IR/\IZ)^{2\dim(A/S)} = \IR^{2\dim(A/S)}.
\end{equation*}
The subvariety $X \subseteq A$ is called \textit{non-degenerate} if and only if there exists a point $x_0 \in (X^{\mathrm{sm}})_{\IC_\nu}^{\mathrm{an}} \cap \pi^{-1}(U)$ such that $\mathrm{rank}_{\mathrm{Betti}}(X,x_0) = 2\dim(X)$. It is easy to see that these definitions depend neither on the choice of $U$ nor $a$. A priori, whether a subvariety is non-degenerate or not may depend on the choice of archimedean place $\nu$, but this is not the case by Gao's algebraic characterization of degenerate subvarieties \cite[Theorem 1.1]{Gao2018a}. Further below, we exclusively work with a fixed archimedean place $\nu$ and, for the purposes of this article, degeneratedness could be also simply understood as with respect to this place.

With these preparations, we can state the first conjecture studied in this article.

\textbf{Relative equidistribution conjecture (REC).} \textit{Let $X \subseteq A$ be a non-degenerate geometrically irreducible subvariety. For each place $\nu \in \Sigma_\infty(K)$, there exists a measure $\mu_{\nu}$ on $X_{\IC_\nu}^{\mathrm{an}}$ with the following property: If $(x_i)$ is an $X$-generic subsequence of closed points in $X$ such that $\hhat(x_i) \rightarrow \hhat(X)$, then
\begin{equation}
\label{equation:equidistribution}
\frac{1}{\# \mathbf{O}_\nu(x_i)}\sum_{x \in \mathbf{O}_\nu(x_i)} f(x)\longrightarrow \int_{X_{\IC_\nu}^{\mathrm{an}}} f \mu_{\nu}, 
\ i \rightarrow \infty,
\end{equation}
for every compactly supported continuous function $f: X_{\IC_\nu}^{\mathrm{an}} \rightarrow \IR$.}

If the base variety $S$ in (REC) is a single point, the conjecture specializes to the classical equidistribution conjecture, which is a result of Szpiro, Ullmo, and Zhang \cite{Szpiro1997} if the place $\nu$ is archimedean. In addition, Yuan's bigness theorem \cite{Yuan2008} implies (REC) in the case of non-archimedean places $\nu \in \Sigma_f(K)$. The non-degeneracy condition is always satisfied in the classical case, but some condition is definitely needed in the general case. Indeed, if $X$ is the total space of a trivial family $A = A_0 \times \IP^1_\IQ$ whose fiber is an abelian variety $A_0$ over $\IQ$, then one can easily pick an infinite sequence of torsion points in the fiber over each $q \in \IP^1(\IQ)$. For each of these sequences, the Galois orbits of their elements are all contained in the same fiber. Combining them, one can obtain generic sequences of closed points in $A$ whose averaged Galois orbits cannot converge weakly to any fixed measure. Indeed, one has complete control over the associated pushforward measures on $\IP^1(\IQ)$.

The author cannot rule out that (REC) is empty if $\hhat(X)>0$, as there may be no non-degenerate subvariety $X \subset A$ enjoying an $X$-generic sequence $(x_i)$ with $\hhat(x_i) \rightarrow \hhat(X)>0$. 
However, there are plenty of non-degenerate subvarieties $X \subseteq A$ containing a Zariski-dense set of torsion points, namely all non-degenerate subvarieties $X$ with $\dim(X)=\dim(A)-\dim(S)$. A well-studied example are the sections of families of elliptic curves \cite{Corvaja2022}. In stark contrast to the case $\dim(S)=0$, these do not all necessarily come from subgroup schemes of $A$, all of which can be formally seen to have zero height. By the following theorem, the said non-degenerate subvarieties also satisfy $\hhat(X)=0$ and hence provide non-trivial examples of (REC). To simplify our exposition, we avoid the more technical case of non-archimedean places $\nu \in \Sigma_f(K)$. In the following, the symbol $\nu$ hence always denotes an archimedean place. 
\begin{theorem}
\label{theorem:equidistribution}
Let $X \subseteq A$ be a geometrically irreducible subvariety with $\pi(X)=S$ that is non-degenerate and contains an $X$-generic subsequence $(x_i)$ of points with $\hhat(x_i) \rightarrow 0$. Then, we have
\begin{equation*}
\hhat(X) = \lim_{k \rightarrow \infty} \left(\frac{ h_{\overline{\mathcal{O}}(1)}(\overline{X}_k)}{n^{2k}} \right) = 0
\end{equation*}
and, for each archimedean place $\nu \in \Sigma_\infty(K)$, there exists a measure $\mu_\nu$ on $X_{\IC_\nu}^{\mathrm{an}}$  such that \eqref{equation:equidistribution} holds for every continuous function $f: X_{\IC_\nu}^{\mathrm{an}} \rightarrow \IR$ with compact support.
\end{theorem}

This theorem generalizes a result of DeMarco and Mavraki \cite[Corollary 1.2]{DeMarco2020}, though we impose an additional compact support assumption on the test functions here and restrict to archimedean places. In fact, their result amounts to the special case where $A$ is a fiber product of elliptic families over $S$ and $X$ is the image of a section of $\pi: A \rightarrow S$. (Related results can be found in \cite{Corvaja2022}.) Their proof relies on results of Silverman \cite{Silverman1992,Silverman1994a,Silverman1994} controlling the local Néron-Tate height on $X$ near the boundary $\overline{X} \setminus X$ where $\overline{X}$ is the Zariski closure of $X$ in a certain compactification of $A$. Silverman's explicit results allow to verify directly that the restriction of the limit Néron-Tate height to $X$ can be identified with a semipositive Arakelov height on $\overline{X}$ (see \cite[Theorem 1.1]{DeMarco2020}). Let us mention that our height $\hhat(X)$ coincides with their height $h_P(B)$.

If $\dim(X)\geq 2$, this approach seems less feasible as the Néron-Tate local heights generally exhibit singularities on compactifications (e.g., on toroidal compactifications). Unfortunately, the appearance of these singularities does not seem to be well-recorded in the literature. The degeneration of the archimedean Néron-Tate local heights and their singularities in the case $X=A$ is however well-understood \cite{BurgosGil2018, BurgosGil2016b}. 

Independent from the present work of the author, Yuan and Zhang have recently proven a more comprehensive equidistribution result \cite[Theorems 5.4.3 and 6.2.3]{Yuan2021}, which supersedes the above theorem and confirms the conjecture (REC) more generally for bounded test functions $f: X^{\mathrm{an}}_{\IC_\nu} \rightarrow \IR$. In addition, Gauthier has also proven a more general equidistribution result \cite[Theorem 2]{Gauthier2021}, and both these results cover also non-archimedean places. They also establish the existence of the measure $\mu_\nu$ independently from the existence of a sequence $(x_i)$ as in Theorem \ref{theorem:equidistribution}, which our shorter proof does not. However, it is a feature of our proof as well that the measure $\mu_\nu$ is the same for any other sequence $(x_i^\prime)$ satisfying the conditions of the theorem.


Let us briefly describe the strategy employed in the proof of Theorem \ref{theorem:equidistribution}. The main idea is to apply equidistribution techniques not directly to any canonical limit height, but instead to the standard projective heights $h_{\overline{\mathcal{O}}(1)}$ on the subvariety $\overline{X}_k \subseteq \IP^N_{K}$ for some sufficiently large $k$. For technical reasons, we actually work with the Zariski closure $\overline{Y}_k \subset \IP^N_{K} \times \IP^N_{K}$ of the graph of the map $[n^k]|_X: X \rightarrow [n^k](X)$. This circumvents the problem that $[n^k]|_{X}$ may not be an isomorphism. For the description of our method here, we simply assume that we are in this case and continue with $\overline{X}_k$ instead of $\overline{Y}_k$. The $X$-generic sequence $(x_i)$ determines an $\overline{X}_k$-generic sequence $(x_i^{(k)})$ by setting $x_i^{(k)} = \iota([n^k](x_i))$. 

For equidistribution, we have to control the following two quantities:
\begin{itemize}
	\item[(i)] the projective height $h_{\overline{\mathcal{O}}(1)}(\overline{X}_k)$, and
	\item[(ii)] the limit superior $\mathcal{l}_k := \limsup_{i \rightarrow \infty}(h_{\overline{\mathcal{O}}(1)}(x_i^{(k)})/n^{2k})$.
\end{itemize}
By means of Zhang's inequalities \cite[Theorem 5.2]{Zhang1995}, it is easy to get some control on (i) once one controls (ii): The projective height of any point on $\overline{X}_k$ is non-negative, whence $h_{\overline{\mathcal{O}}(1)}(\overline{X}_k) \geq 0$. In the other direction, we have $$h_{\overline{\mathcal{O}}(1)}(\overline{X}_k) \leq \mathcal{l}_k n^{2k}$$ as $(x_i^{(k)})$ is an $\overline{X}_k$-generic sequence. 

Thus it suffices to control (ii). Here again, it is clear that $\mathcal{l}_k \geq 0$ so that we are only in need of an upper bound. For a fixed integer $k \geq 1$, we derive such an estimate in Lemma \ref{lemma::smallsequence} from a bound on
\begin{equation}
\label{equation::uniform}
\left|\frac{h_{\overline{\mathcal{O}}(1)}(x_i^{(k)})}{n^{2k}}-\hhat(x_i)\right|
\end{equation}
that is uniform in $i \geq 1$. To explain the origin of such a bound, let us choose an arbitrary immersion $\kappa: S \hookrightarrow \IP^{M}_K$ of the base variety and an associated projective height $h$ on $S$. Using a result of Manin and Zarhin \cite{Zarhin1972}, which we invoke in the more recent version of \cite[Corollary 7.4]{Silverman1987}, the quantity \eqref{equation::uniform} is bounded from above by
\begin{equation*}
c_1 \cdot n^{-2k} \max \{1, h(\pi(x_i))\}
\end{equation*}
where the constant $c_1 > 0$ depends neither on $i$ nor on $k$ (Lemma \ref{lemma::gaohabegger2}). This bound tends to $0$ uniformly in $i$ as soon as $h(\pi(x_i))$ can be uniformly bounded for all $i \geq 1$. This is precisely what a height bound due to Dimitrov, Gao, and Habegger \cite[Theorem 1.6]{Dimitrov2021a} provides (Theorem \ref{lemma::gaohabegger}). These estimates give us an equidistribution result for $(x_i^{(k)})$ on $\overline{X}_k$ -- up to a certain error term that reflects our incomplete control on both (i) and (ii). This translates directly to a near-equidistribution result for $(x_i)$ on $X$. With sufficient bookkeeping, all error terms can be seen to disappear as $k \rightarrow \infty$, and we obtain \eqref{equation:equidistribution} asymptotically. 

It should be mentioned that a previous article \cite{Kuehne2022} of the author contains an asymptotic approach in order to prove the equidistribution conjecture for semiabelian varieties, which does not follow directly from Yuan's equidistribution theorem \cite{Yuan2008}. Apart from the general idea to extend the reach of Yuan's theorem through asymptotics and additional arguments tailored to the specific situation at hand, the approach here is rather orthogonal to the one in \cite{Kuehne2022}. In fact, the canonical heights on semiabelian varieties extend well to the boundary of good compactifications. Yuan's theorem can hence be applied, although it does not yield anything if the chosen compactification of the given semiabelian variety has strictly negative height, which generally cannot be avoided. In contrast, the conditions of Yuan's theorem are violated here because the canonical heights have singularities in general. This accounts for the fact that all arguments in \cite{Kuehne2022} invoke canonical heights, whereas only approximating projective heights are used in this article. As a consequence, the common intersection of both arguments is rather narrow, although the initial idea is identical.




A particular feature of the classical equidistribution conjecture is its close relation with the Bogomolov conjecture. In fact, Ullmo \cite{Ullmo1998} and Zhang \cite{Zhang1998} deduced the Bogomolov conjecture from the equidistribution conjecture. Our application of Theorem \ref{theorem:equidistribution} follows their footsteps. In contrast to the classical setting of \cite{Ullmo1998,Zhang1998}, it is necessary in our relative setting to establish that the subvarieties used are non-degenerate so that it is legitimate to use our equidistribution result. For this, a consequence \cite[Theorem 1.3]{Gao2018a} (see also the corrigendum \cite{Gao2021b}) of Gao's work on the mixed Ax-Schanuel conjecture \cite{Gao2018} is a crucial ingredient that yields immediately what is needed.

The main consequence of our approach being able to work with families of abelian varieties instead of single ones is that it allows us to improve classical results of Manin-Mumford and Bogomolov type to uniform results. As a first striking consequence we obtain the following strong version of the Manin-Mumford conjecture. Note that the uniformity in the genus has not been achieved by any of the other methods proposed up to the present day \cite{Hrushovski2001, Pila2008, Raynaud1983, Ullmo1998}. However, uniformity has been achieved under additional restrictions on the Jacobian in a few particular cases. In the special case that $\mathrm{Jac}(C)$ polarized by the theta divisor $\Theta_C$ is the self-product of a polarized elliptic curve, the assertion of the theorem below was proven by David and Philippon \cite[Théorème 1.13]{David2007}. There is also a noteworthy result of DeMarco, Krieger, and Ye \cite[Theorem 1.1]{DeMarco2020a}. They showed a similar assertion for complex algebraic curves whose Jacobian is an abelian surface admitting real multiplication by the real quadratic order of discriminant $4$ (compare Section 9 of \cite{DeMarco2020a}). Apart from these results, the author is only aware of a result due to Dimitrov, Gao, and Habegger \cite[Proposition 8.1]{Dimitrov2021a}, which requires $C$ to be defined over a number field and $\mathrm{Jac}(C)$ to have sufficiently large ``modular height''. We discuss their work in detail further below.

\begin{theorem}[uniform Manin-Mumford conjecture]
	\label{theorem::uniform_mm}
	For each $g \geq 2$, there exists an integer $c_2(g) \geq 1$ such that the following assertion is true: For every smooth proper genus $g$ curve $C$ defined over $\IC$ and every divisor $D$ of degree $1$ on $C$, we have
	\begin{equation*}
	\# (\iota_D(C) \cap \mathrm{Tors}(\mathrm{Jac}(C))) \leq c_2(g)
	\end{equation*}	
	where
	\begin{equation*}
	\iota_D: C \longrightarrow \mathrm{Jac}(C), \ q \longmapsto [q] - D,
	\end{equation*}
	is the Abel-Jacobi embedding sending $D$ to the identity in $\mathrm{Jac}(C)$.
\end{theorem}

Whereas the number of torsion points on $\iota_D(C) \subseteq \mathrm{Jac}(C)$ is uniformly bounded by the theorem, there can be no uniform bound on their order, even under the additional restriction $D=[p]$ with a closed point $p \in C$ (see \cite[Section 9.5]{DeMarco2020a} for a counterexample to such a strengthening).

We actually show a stronger Bogomolov-type result in this article. To state it, we need to introduce some additional terminology. \textit{Details are given at the end of this introduction. We restrict ourselves to moduli over $\IQbar$.} For an integer $n \geq 3$, let $\mathcal{A}_{g,n}$ denote the fine moduli space of principally polarized abelian varieties with level $n$ structure and write $\pi_{g,n}: \mathcal{B}_{g,n} \rightarrow \mathcal{A}_{g,n}$ for the associated universal family of abelian varieties. By $\mathcal{M}_{g,n}$ we denote the fine moduli space of smooth proper genus $g$ curves with a Jacobi structure of level $n$. We let $\tau_{g,n}: \mathcal{M}_{g,n} \rightarrow \mathcal{A}_{g,n}$ be the Torelli map with level $n$ structure. Both $\mathcal{A}_{g,n}$ and $\mathcal{B}_{g,n}$ as well as $\mathcal{M}_{g,n}$ are actually quasi-projective varieties over $\IQbar$, and $\tau_{g,n}: \mathcal{M}_{g,n} \rightarrow \mathcal{A}_{g,n}$ is a morphism between $\IQbar$-varieties.

For every $s \in \mathcal{M}_{g,n}(\IQbar)$, all points $q \in  \mathcal{C}_{g,n,s}(\IQbar)$ and any divisor $D$ of degree $1$ on $\mathcal{C}_{g,n,s}$, the difference $[q]-D$ can be considered as a $\IQbar$-point of the Jacobian $\mathrm{Jac}(\mathcal{C}_{g,n,s})$, which is canonically a subvariety of $\mathcal{B}_{g,n}$. Having chosen an immersion $\iota: \mathcal{B}_{g,n} \hookrightarrow \IP^N_{\IQbar}$, we can hence assign a height $\hhat([q]-D)$ to this difference.

\begin{theorem}[uniform Bogomolov conjecture]
	\label{theorem::bogo}
	Let $g \geq 2$ and $n \geq 3$ be integers, and let $\iota: \mathcal{B}_{g,n} \hookrightarrow \IP^N_{\IQbar}$ be an immersion. There exist constants $c_3 = c_3(g,n,\iota)$, $c_4 = c_4(g,n,\iota)>0$ for which the following assertion is true: For each $s \in \mathcal{M}_{g,n}(\IQbar)$ and every divisor $D$ of degree $1$ on the curve $\mathcal{C}_{g,n,s}$, we have
	\begin{equation}
	\label{equation::my_equation}
	\# \left\{ q \in \mathcal{C}_{g,n,s}(\IQbar) \ \middle| \ \hhat([q]-D) \leq c_3 \right\} < c_4.
	\end{equation}
\end{theorem}

In Section \ref{section::uniformity}, we deduce this theorem from Proposition \ref{proposition:bogo}, which is an analogue assertion on more general families of smooth proper algebraic curves of genus $g$ instead of universal ones. Let us remark that conversely this proposition could be also easily deduced from Theorem \ref{theorem::bogo}. However, the proposition is amenable to a proof by induction on the dimension of the base and we rely on this fact for our proof.

A related uniform bound on the essential height minima of smooth algebraic curves of a fixed genus $g$ has been published recently by Wilms \cite[Corollary 1.5]{Wilms2021}. However, it seems not yet to imply our proposition because a uniform treatment of points of height below the essential minimum is needed. 
More recently, Looper, Silverman, and Wilms \cite{Looper2021} have obtained a proof of the function field case of Theorems \ref{theorem::uniform_mm} and \ref{theorem::bogo} with explicit constants, using the same approach as in \cite{Wilms2021}. In any case, the methods used in their approach and the one here seem rather disjoint. Wilms' argument uses explicit height computations and comparisons in the setting of Jacobian embeddings, whereas our Proposition \ref{proposition:bogo} is not restricted to this setting. It should also be said that we opted for a simpler proof at the cost of introducing some additional assumptions in the proposition (e.g., the existence of a principal polarization). The reader interested in the most general result is referred to meanwhile finished work of Gao, Ge, and the author \cite{Gao2021a}.

Our Theorem \ref{theorem::bogo} should be also compared with a recent result by Dimitrov, Gao, and Habegger \cite[Proposition 7.1]{Dimitrov2021a}. In the setting of the theorem, they proved that there exist constants $c_5 = c_5(g,n,\iota)$, $c_6= c_6(g,n,\iota)$, $c_7 = c_7(g,n,\iota)$, $c_8 = c_8(g,n,\iota)>0$, a projective compactification $\overline{\mathcal{A}}_{g,n}$ of $\mathcal{A}_{g,n}$, and an ample line bundle $M$ on $\overline{\mathcal{A}}_{g,n}$ with an associated Weil height $h_M: \overline{\mathcal{A}}_{g,n}(\IQbar) \rightarrow \IR$ such that the following is true: For each $s \in \mathcal{M}_{g,n}(\IQbar)$ such that $h_M(\tau_{g,n}(s)) \geq c_5$, there exists a finite subset $\Xi_s\subseteq \mathcal{C}_{g,n,s}(\IQbar)$ of cardinality $\leq c_6$ such that
\begin{equation}
\label{equation::their_equation}
\# \left\{ q \in \mathcal{C}_{g,n,s}(\IQbar) \ \middle| \ \hhat(q-p) \leq c_7 \cdot h_M(\tau_{g,n}(s)) \right\} < c_8
\end{equation}
for all $p \in \mathcal{C}_{g,n,s}(\IQbar) \setminus \Xi_s$. (It should be said that \cite{Dimitrov2021a} uses a specific compactication endowed with an ample line bundle, but this is immaterial by \cite[Proposition 2.3]{Vojta1999}.) Whenever the ``modular height'' $h_M(\tau_{g,n}(s))$ is large enough, this implies a uniform bound on $\# (\iota_{[p]}(C) \cap \mathrm{Tors}(\mathrm{Jac}(C)))$. However, the argument in \cite{Dimitrov2021a} yields nothing if $h_M(\tau_{g,n}(s))<c_5$.

It has been already implicit in \cite{Dimitrov2021,Dimitrov2021a,Dimitrov2022} that a combination of the cardinality bounds in \eqref{equation::my_equation} and \eqref{equation::their_equation} can be used to obtain a uniform version of the Mordell-Lang conjecture \cite[Conjecture 1.1]{Dimitrov2022} through Rémond's explicit versions of Mumford's \cite{Mumford1965a, Remond2000a} and Vojta's inequality \cite{Remond2000, Vojta1991} (see \cite[Section 2]{Dimitrov2021}). As a question, it seems to be originally due to Mazur (see the question on \cite[p.\ 234]{Mazur1986} as well as the discussion on \cite[p.\ 223]{Mazur2000}).


\begin{theorem}[uniform Mordell-Lang conjecture]
	\label{theorem::uniform_ml}
	For each $g \geq 2$, there exists a constant $c_9(g) > 0$ such that the following assertion is true: For every smooth proper genus $g$ curve $C$, every divisor $D$ of degree $1$ on $C$, and every subgroup $\Gamma \subset \mathrm{Jac}(C)(\IC)$ of finite rank $\varrho$, we have
	\begin{equation}
	\label{equation::my_equation2}
	\# (\iota_D(C) \cap \Gamma) \leq c_9(g)^{1+\varrho}
	\end{equation}	
	where
	\begin{equation*}
	\iota_D: C \longrightarrow \mathrm{Jac}(C), \ q \longmapsto [q] - D,
	\end{equation*}
	is an Abel-Jacobi embedding.
\end{theorem}

The choice $\Gamma=\mathrm{Tors}(\mathrm{Jac}(C))$ leads back to Theorem \ref{theorem::uniform_mm}. Choosing instead $\Gamma=\mathrm{Jac}(C)(K)$ for a smooth algebraic curve $C$ defined over $K$ with Mordell-Weil rank $\varrho$, we obtain the rather uniform bound
\begin{equation*}
	\# C(K) \leq c_{10}(g)^{1+\varrho}
\end{equation*}
on the number of rational points. A weaker bound 
\begin{equation*}
	\# C(K) \leq c_{11}(g,[K:\IQ])^{1+\varrho}
\end{equation*}
has been already obtained by Dimitrov, Gao, and Habegger \cite{Dimitrov2021a}. It is not surprising that their proof rests on \eqref{equation::their_equation}. Furthermore, Stoll \cite{Stoll2019} used the Chabauty-Coleman method \cite{Chabauty1941, Coleman1985} to prove a bound of the form
\begin{equation}
\label{equation::their_equation3}
\# C(K) \leq c_{12}(g,[K:\IQ])
\end{equation}
if $C$ is a hyperelliptic curve of genus $g$ with Mordell-Weil rank $\varrho \leq (g - 3)$. The condition that $C$ is a hyperelliptic curve was lifted later by Katz, Rabinoff, and Zureick-Brown \cite{Katz2016}. Without any assumption on the Mordell-Weil rank $\varrho$, they also obtained a uniform bound on the number of \textit{$K$-rational} torsion points contained in $\iota_D(C)$, which was already superseded by the results of \cite{Dimitrov2021a}. The bounds in \cite{Katz2016,Stoll2019} are very explicit and rather small. In contrast, it seems still an open problem to prove that the constant $c_9(g)$ in our theorem is effectively computable.

We do not give a deduction of Theorem \ref{theorem::uniform_ml} from the cardinality bounds \eqref{equation::my_equation} and \eqref{equation::their_equation} since there would be nothing to add to \cite{Dimitrov2021,Dimitrov2021a,Dimitrov2022} or \cite{Pazuki2021}. In fact, the proof is a straightforward adjustment of the proof of \cite[Proposition 8.1]{Dimitrov2021a} from \cite[Proposition 7.1]{Dimitrov2021a}. The reader can find this adapted proof in \cite[Section 2.4]{Dimitrov2022} and consult \cite[Section 3]{Dimitrov2022} for the specialization argument from $\IC$ to $\IQbar$. In fact, a combination of \eqref{equation::my_equation} and \eqref{equation::their_equation} yields the assertion of Proposition 2.5 in \textit{loc.}\ \textit{cit.}\ unconditionally (i.e., without assuming the relative Bogomolov conjecture). This renders the proof of Conjecture 1.1 in \textit{loc.}\ \textit{cit.}, which corresponds to our Theorem \ref{theorem::uniform_ml}, unconditional. The proof rests on previous work of (in chronological order) Mumford \cite{Mumford1965a}, Vojta \cite{Vojta1991}, Bombieri \cite{Bombieri1990}, Faltings \cite{Faltings1991}, Rémond \cite{Remond2000a,Remond2000}, as well as Dimitrov, Gao, and Habegger \cite{Dimitrov2021,Dimitrov2021a}. Furthermore, Gao has meanwhile written a comprehensive survey \cite{Gao2021}. 

The author's original motivation to prove Theorem \ref{theorem:equidistribution} was the relative Bogomolov conjecture \cite[Conjecture 1.2]{Dimitrov2022}, which is the Bogomolov-type analogue of Pink's relative Manin-Mumford conjecture \cite[Conjecture 6.2]{Pink2005a} and is also related to a conjecture in Zhang's ICM talk \cite{Zhang1998a}. In fact, Theorem \ref{theorem:equidistribution} can be used to affirm this conjecture for all subvarieties in fibered products of elliptic families, generalizing \cite[Theorem 1.4]{DeMarco2020}. The reader can find details in the separate article \cite{Kuehne2023a}.

Finally, it should be said that the results of both Bogomolov type as well as of Mordell-Lang type presented here have been extended beyond curves in a recent joint work with Gao and Ge \cite{Gao2021a}. In the case of curves, Yuan \cite{Yuan2021} also obtained improved versions of our Theorems \ref{theorem::uniform_mm} and \ref{theorem::bogo} by a different approach.

\textbf{Notation and conventions.} \textit{General.}
For two terms $a$ and $b$, we write $a \ll b$ if there exists a positive real number $c$ such that $a \leq c \cdot b$. If $c$ depends on some data, say an algebraic variety $X$, we write $a \ll_X b$ etc. We use $\gg$ similarly. In each of the following sections, we may suppress the dependence of constants on some basic data for readability. Where this is done, it is indicated at the beginning of the respective section.

\textit{Number fields.} Throughout this article, we let $K$ denote a number field with integer ring $\caO_K$. In addition, $\Sigma_f(K)$ (resp.\ $\Sigma_\infty(K)$) is the set of non-archimedean (resp.\ archimedean) places, and we set $\Sigma(K) = \Sigma_f(K) \cup \Sigma_\infty(K)$. 
By $\IC_\nu$ is denoted a completion of an algebraic closure $\overline{K}_\nu$ of $K_\nu$, 
and by $p_\nu$ the residue characteristic of $\IC_\nu$. For all $\nu \in \Sigma_f(K)$, the absolute value $\vert \cdot \vert_\nu$ on $\IC_\nu$ is normalized such that $\vert p_\nu \vert_\nu = p_\nu^{-[K_\nu:\IQ_{p_\nu}]}$. 
We use the standard absolute values on $\IR$ and $\IC$ for the archimedean places. This normalization leads to an additional factor
\begin{equation*}
\delta_\nu =
\begin{cases}
2 & \text{if $\nu$ is complex archimedean,} \\
1 & \text{otherwise},
\end{cases}
\end{equation*}
in some identities.

\textit{Algebraic geometry (General).} Denote by $k$ an arbitrary field. A \textit{$k$-variety} is a separated, reduced scheme of finite type over $k$. By a \textit{subvariety} of a $k$-variety we mean a closed reduced subscheme. A subvariety is determined by its underlying topological space and we frequently identify both. The tangent bundle of a $k$-variety $X$ is written $TX$ and its fiber over a point $x \in X$ is denoted by $T_x X$. Furthermore, $X^{\mathrm{sm}}$ denotes the smooth locus of $X$. If $X$ is an irreducible $k$-variety, we write $\eta_X$ for its generic point. 

For a non-negative integer $d$ and a $k$-variety $X$, a \textit{$d$-cycle} on $X$ is a finite formal sum $\sum_{i=1}^{r} n_i [Z_i]$ where each $n_i$ is an integer and each $Z_i$ is a $k$-irreducible subvariety of $X$ having dimension $d$. 


For a map $f: X \rightarrow Y$ between algebraic varieties and a point $y \in Y$ with residue field $k(y)$, we write $X_y$ for the fiber $X \times_Y \mathrm{Spec}(k(y))$ over $y$. We use a similar notation for $S$-points $y \in Y(S)$, $S$ an arbitrary scheme.

\textit{Generic sequences.} Let $X$ be an irreducible algebraic $k$-variety. We say that a sequence $(x_i) \in X^\IN$ of closed points is \textit{$X$-generic} if none of its subsequences is contained in a proper algebraic subvariety of $X$. If the variety $X$ can be inferred from context, we simply say \textit{generic} instead of $X$-generic.

\textit{Line bundles and intersection theory.} For line bundles $L_1, L_2, \dots, L_d$ on a proper algebraic variety $X$ of dimension $d$ over a field $k$, we use the intersection numbers
\begin{equation*}
L_1 \cdot L_2 \cdots L_d \in \IZ
\end{equation*}
defined by Kleiman \cite{Kleiman1966} and Snapper \cite{Snapper1960} (see \cite[Section VI.2]{Kollar1996} for a good introduction). These coincide with the numbers
\begin{equation*}
\deg(c_1(L_1) \cap c_1(L_2) \cdots \cap c_1(L_d) \cap [X]) \in \IZ
\end{equation*}
in the terminology of \cite{Fulton1998}. If $\{ M_1, M_2, \dots, M_r \} = \{ L_1, L_2, \dots, L_d \}$ and each $M_i$ occurs precisely $n_i$ times among $L_1,L_2,\dots,L_d$, we set
\begin{equation*}
M_1^{n_1} \cdot M_2^{n_2} \cdots M_r^{n_r} := L_1 \cdot L_2 \cdots L_d.
\end{equation*}
A similar notation is used for the arithmetic intersection numbers defined in Subsection \ref{subsection::arithmeticintersectionnumbers}. Furthermore, we write $\deg_L(X)$ for $L^{d}$. 

The group law of Picard groups of line bundles, as well as of their arithmetic analogue introduced in Subsection \ref{subsection::hermitian}, is written additively. 

\textit{Continuous functions.} 
We use $\mathscr{C}^0$ as an abbreviation for continuous. For any topological space $X$, $\mathscr{C}^0(X)$ denotes the real-valued continuous functions on $X$ and $\mathscr{C}^0_c(X)$ the real-valued continuous functions on $X$ having compact support. We use the analogous notation $\mathscr{C}^\infty$ for the real-valued smooth functions on complex spaces that are introduced below.

\textit{Tangent spaces.}
For each differentiable or real-analytic manifold $M$ we denote by $TM$ its tangent bundle. The fiber of $TM$ over $x \in M$ is denoted $T_x M$.

Let $M$ be a complex manifold (e.g., $(X^{\mathrm{sm}})^{\mathrm{an}}_{\IC_\nu}$  for an algebraic variety $X$ over $K$ and some $\nu \in \Sigma_\infty(K)$). To $M$ is associated its real tangent bundle $T_{\IR} M$ and its holomorphic tangent bundle $T^{1,0}_{\IC} M$ (e.g., the analytification $(T X)^{\mathrm{an}}$ of the tangent bundle of a smooth complex algebraic variety $X$). The reader is referred to \cite[Section 0.2]{Griffiths1994} and \cite[Section 1.2]{Huybrechts2005} for details.

\textit{Complex spaces.} Let $M$ be a reduced complex (analytic) space (e.g., $X^{\mathrm{an}}_{\IC_\nu}$ for a $K$-variety $X$ and some $\nu \in \Sigma_\infty(K)$). Recall that this means that $M$ is locally biholomorphic to a closed analytic subvariety $V$ of a complex domain $U \subset \IC^n$. A function $f$ on $M$ is \textit{smooth} if, for each such sufficiently small local chart, it is the restriction of a smooth function on $U$. We write $\mathscr{C}^\infty(M)$ for the smooth real-valued functions on $M$. In the same way, we use local charts to define \textit{plurisubharmonic} functions on $M$ as restrictions of plurisubharmonic functions. 

Similarly, a smooth form $\omega$ on $M$ is a differential form on the smooth locus $M^\mathrm{sm}$ of $M$ with the following extension property: $M$ can be covered by local charts $V \subset U \subset \IC^n$ as above such that for each chart the differential form $\omega|_{V^{\mathrm{sm}}}$ is the restriction of a $\mathscr{C}^\infty$-differential form on $U$. There are also well-defined linear operators $d$ and $d^c = i/2\pi(\delbar - \del)$ on the $\mathscr{C}^\infty$-differential forms on $M$.\footnote{\label{footnote::dc}Note that there are different definitions of $d^c$ in the literature. Many texts on Arakelov theory (e.g.\ \cite{Gubler2003,Moriwaki2014,Soule1992}) use $d^c=i/4\pi(\overline{\partial}-\partial)$.} For each local chart $V \subset U \subset \IC^n$, these are simply the restrictions of the operators of the same name on $\IC^n$. Recall that for every reduced complex space $M$ of dimension $d$ and every smooth $(d,d)$-form $\omega$, the integral $\int_{M} \omega$ is well-defined by classic work of Lelong \cite{Lelong1957}. In fact, this integral is by definition the ordinary integral $\int_{M^{\mathrm{sm}}} \omega$ over a(n in general non-compact) complex manifold $M^{\mathrm{sm}}$, whose finiteness is proven in \cite{Lelong1957}.

We remark that the usual rules for integration on complex manifolds still apply for reduced complex spaces and we use them without further comment. In particular, let $f: M \rightarrow N$ be a finite branched holomorphic covering of pure order $q$ (see \cite[Chapter C]{Gunning1990a} for definitions) between reduced complex analytic spaces of (necessarily the same) dimension $d$: Let $\eta$ (resp.\ $\eta^\prime$) be a smooth $(d,d)$-form on $M$ (resp.\ $N$) such that $\eta = f^\ast \eta^\prime$. Then, 
\begin{multline}
	\label{equation::substitution_rule}
	\int_{M} \eta = \int_{M^{\mathrm{sm}}} \eta = \int_{M^\mathrm{sm} \cap f^{-1}(N^{\mathrm{sm}})} \eta = q \int_{N^{\mathrm{sm}} \cap f(M^{\mathrm{sm}})} \eta^\prime = q \int_{N^{\mathrm{sm}}} \eta^\prime = q \int_{N} \eta^\prime.
\end{multline}
Here, the equalities follow from the definitions, discarding or adding a null set, and the ordinary substitution rule in the smooth setting. 


\textit{Moduli spaces of abelian varieties}. We work with certain moduli functors on schemes over $\IQbar$. Note that for every abelian schemes $\pi: A \rightarrow S$, there exists a dual abelian scheme $A^\vee$ by \cite[Corollary 6.8]{Mumford1994}.\footnote{Note that \cite[Chapters 6 and 7]{Mumford1994} work only with locally noetherian schemes. However, this can be dropped as every abelian scheme is the base change of an abelian scheme over a $\IZ$-algebra of finite type (compare \cite[footnote on p.\ 76]{Mumford1967}).} With a scheme $S$ over $
\IQbar$, we associate similar to \cite[Definition 7.2]{Mumford1994} the set $\mathcal{A}_{g,d,n}(S)$ consisting of triples $(\pi, \lambda, \sigma)$ where
\begin{enumerate}
	\item[(i)] $\pi: A \rightarrow S$ is an abelian scheme of relative dimension $g$;
	\item[(ii)] $\lambda: A \rightarrow A^\vee$ is a polarization (\cite[Definition 27.280]{Goertz2023}) of degree $d^2$ (i.e., $\lambda_\ast 
	\mathcal{O}_A$ is locally free of rank $d^2$);
	\item[(iii)] $\sigma: A[n]  \xrightarrow{\sim} (\IZ/n\IZ)_S^{2g}$ is an isomorphism of finite flat group schemes over $S$ that respects the sympletic structures on both sides. Here $A[n]$ is endowed with the symplectic structure induced by the Weil pairing whereas  $(\IZ/n\IZ)_S^{2g}$ is endowed with the standard sympletic structure.\footnote{Note that the level structure in \cite{Mumford1994} disregards sympletic structures. We do not follow this here as the sympletic structure fits with the Jacobi structures used for moduli of curves in \cite{Deligne1969}.}
\end{enumerate}
In the obvious way, this assignment is functorial. If $n \geq 3$, it is a smooth quasi-projective variety (see \cite[Theorem 7.9]{Mumford1994} and the ``lemma of Serre'' \cite{Serre1962}). The associated universal family of abelian varieties is written $\pi_{g,d,n}: \mathcal{B}_{g,d,n} \rightarrow \mathcal{A}_{g,d,n}$. In \cite[Theorem 2.1.11]{Olsson2012}, it is deduced from the results of \cite{Mumford1994} that the functor $\mathcal{A}_{g,d,n}$ is still a smooth Deligne–Mumford stack over $\IQbar$. We write $\mathcal{A}_{g,n}$, $\mathcal{B}_{g,n}$, $\pi_{g,n}$ instead of $\mathcal{A}_{g,1,n}$, $\mathcal{B}_{g,1,n}$, $\pi_{g,1,n}$.

The kernel $\mathrm{ker}(\lambda)$ of any polarization $\lambda: A \rightarrow A^\vee$ of degree $d^2$ is a finite locally free group scheme (\cite[Corollary 27.177]{Goertz2023}). If $S$ is connected, there exists a unique tuple
\begin{equation*}
	\boldsymbol{\Delta}=(\delta_1,\dots,\delta_g)
\end{equation*}
with positive integers $\delta_i$ satisfying $\delta_i \mid \delta_{i+1}$ ($1 \leq i < g$) such that, for every geometric point $\overline{s}$ of $S$, the kernel of $\lambda_{\overline{s}}: A_{\overline{s}}\rightarrow A^\vee_{\overline{s}}$ is isomorphic to the constant group scheme $\prod_{i=1}^g (\IZ/\delta_i\IZ)_{\overline{s}}^2$. This follows, for example, by combining \cite[Lemma 2.1.9]{Olsson2008}, \cite[Corollary of Theorem 1]{Mumford1966} and \cite[Proposition 1]{Mumford1967}. Necessarily, we have $d = \delta_1 \cdot \delta_2 \cdots \delta_g$. We call $\boldsymbol{\Delta}$ the type of the polarization.


For any $\boldsymbol{\Delta}$ as above, we define $\mathcal{A}_{g,\boldsymbol{\Delta},n}$ as the subfunctor of $\mathcal{A}_{g,d,n}$ comprising triples $(\pi,\lambda,\sigma)$ such that the polarization $\lambda$ is of a given type $\boldsymbol{\Delta}$. As the type of a polarization is locally constant by the above, this induces a partition of $\mathcal{A}_{g,d,n}$ into open substacks $\mathcal{A}_{g,\boldsymbol{\Delta},n}$. As $\mathcal{A}_{g,d,n}$ is a smooth quasi-projective variety over $\IQbar$ if $n \geq 3$, each $\mathcal{A}_{g,\boldsymbol{\Delta},n}$ is a smooth quasiprojective variety over $\IQbar$ as well. These varieties are of particular interest for us as their associated complex analytic spaces $\mathcal{A}_{g,\boldsymbol{\Delta},n,
\IC}^{\mathrm{an}}$ have a complex uniformization by the Siegel modular upper half-space
\begin{equation*}
	\mathcal{H}_g = \left\{ \underline{\tau} \in \IC^{g\times g} \ \middle| \ \underline{\tau} = \underline{\tau}^t  \ \text{and} \ \text{$\mathrm{Im}(\tau)\in \IR^{g\times g}$ positive definite} \right\}
\end{equation*}
of degree $g$. Recall that $\mathcal{H}_g$ is endowed with a standard action of $\mathrm{Sp}_{2g}(\IZ)$. In fact, the complex manifold $\mathcal{A}_{g,\boldsymbol{\Delta},n,\IC}^{
\mathrm{an}}$ is the quotient of $\mathcal{H}_g$ by the subgroup  $\Gamma_{\boldsymbol{\Delta}} \cap \Gamma(N) \subseteq \mathrm{Sp}_{2g}(\IZ)$ where
\begin{equation*}
	\Gamma_{\boldsymbol{\Delta}} = \left\{ \mathbf{A} \in \mathrm{Sp}_{2g}(
	\IZ) \mid \mathbf{A} 
	\begin{pmatrix}
	\boldsymbol{0}_{g \times g} & \mathrm{diag}(\boldsymbol{\Delta}) \\
	\mathrm{diag}(\boldsymbol{\Delta}) & \boldsymbol{0}_{g \times g}
	\end{pmatrix} 
	\mathbf{A}^t 
	= 
	\begin{pmatrix}
	\boldsymbol{0}_{g \times g} & \mathrm{diag}(\boldsymbol{\Delta}) \\
	\mathrm{diag}(\boldsymbol{\Delta}) & \boldsymbol{0}_{g \times g}
	\end{pmatrix}
	\right\}
\end{equation*}
and $\Gamma(N)$ is the kernel of the reduction $\mathrm{Sp}_{2g}(\IZ)\rightarrow \mathrm{Sp}_{2g}(\IZ/N\IZ)$. The universal family over $\mathcal{A}^{\mathrm{an}}_{g,\boldsymbol{\Delta},n,\IC}$ descends from the universal family $\pi_{g,\boldsymbol{\Delta}}: \mathcal{A}_{g,\boldsymbol{\Delta}} \rightarrow \mathcal{H}_g$ of complex abelian varieties with polarization type $\boldsymbol{\Delta}$ over the Siegel modular upper half-space defined in \cite[Section 8.1]{Birkenhake2004}. Furthermore, a line bundle $L_{g,\boldsymbol{\Delta}} \rightarrow \mathcal{A}_{g,\boldsymbol{\Delta}}$  inducing the given polarization is constructed in \cite[Section 8.7]{Birkenhake2004}.

\textit{Moduli spaces of curves and the Torelli map}.
Let $\mathcal{M}_{g,n}$ denote the moduli stack over $\IQbar$ parameterizing geometrically irreducible smooth projective curves with Jacobi structure of level $n$ (as in \cite[(5.14)]{Deligne1969}) and write $\pi^\prime_{g,n}:\mathcal{C}_{g,n} \rightarrow \mathcal{M}_{g,n}$ for the universal family over $\mathcal{M}_{g,n}$. If $n \geq 3$, $\mathcal{M}_{g,n}$ is a smooth quasi-projective variety (again by Serre's lemma \cite{Serre1962}). Associating to a curve its Jacobian induces a morphism $\tau_{g,n}: \mathcal{M}_{g,n} \rightarrow \mathcal{A}_{g,n}$ of stacks, the Torelli map with level $n$ structure (compare \cite{Oort1980}). 

\section{Heights}
\label{section:heights}

In this section, we briefly review heights and their basic properties as we need them. We closely follow the notations introduced in \cite[Section 2]{Kuehne2022}, and we refer to there for a detailed exposition of arithmetic intersection theory \cite{Gillet1990} in the setting needed here. Primary sources for arithmetic intersection theory include \cite{ Bost1994, Faltings1991,  Gillet1990, Gubler1998, Gubler2003, Yuan2008, Zhang1995, Zhang1995a}. Furthermore, we mention the introductory textbooks \cite{Moriwaki2014,Soule1992}.


Recall that $K$ is a number field with integer ring $\mathcal{O}_K$. We consider a flat, integral, projective $\mathcal{O}_K$-scheme $\mathcal{X}$ of relative dimension $d$ over $\mathcal{O}_K$. Its generic fiber $X \subseteq \IP^N_K$ is an irreducible, projective $K$-variety. In this setting, we say that $\mathcal{X}$ is a $\mathcal{O}_K$-model of $X$.

\subsection{Hermitian line bundles on arithmetic varieties}
\label{subsection::hermitian}

A $\mathscr{C}^\infty$-hermitian line bundle $\overline{\mathcal{L}}$ on $\mathcal{X}$ is a collection $(\mathcal{L}, \{ \Vert \cdot \Vert_\nu \}_{\nu \in \Sigma_\infty(K)})$ consisting of a line bundle $\mathcal{L}$ on $\mathcal{X}$ and a $\mathscr{C}^\infty$-hermitian metric $\Vert \cdot \Vert_\nu: \mathcal{L} _{\IC_\nu}^{\mathrm{an}} \rightarrow \IR$ for each $\nu \in \Sigma_\infty(K)$. If $K_\nu = \IR$, we assume additionally that the $\nu$-metric is invariant under $\Gal(\IC_\nu/K_\nu)$ (i.e., under complex conjugation on $X_{\IC_\nu}^{\mathrm{an}}$). We say that $\overline{\mathcal{L}}$ is vertically semipositive if $\mathcal{L}$ is relatively nef with respect to $\mathcal{X} \rightarrow \mathrm{Spec}(\mathcal{O}_K)$ (i.e., its restriction to every fiber is nef) and each $\mathscr{C}^\infty$-hermitian metric $\Vert \cdot \Vert_\nu$, $\nu \in \Sigma_\infty(K)$, is semipositive (i.e., its associated Chern form is semipositive\footnote{The reference \cite{Kuehne2022} contains a typo (a superfluous $dd^c$) in the definition of semipositivity.}).

Two hermitian line bundles $\overline{\mathcal{L}}$ and $\overline{\mathcal{M}}$ are called isometric if there is an isomorphism $\mathcal{L} \approx \mathcal{M}$ preserving the metrics at all archimedean places. The arithmetic Picard group $\widehat{\Pic}_{\mathscr{C}^\infty}(\mathcal{X})$ is the set of isometry classes of $\mathscr{C}^\infty$-hermitian line bundles on $\mathcal{X}$. For $\Lotil, \Motil \in \widehat{\Pic}_{\mathscr{C}^\infty}(\mathcal{X})$, we define $\Lotil + \Motil$, $-\Lotil$ as elements of $\widehat{\Pic}_{\mathscr{C}^\infty}(\mathcal{X})$ in the obvious way, obtaining a group structure on $\widehat{\Pic}_{\mathscr{C}^\infty}(\mathcal{X})$. Given a morphism $f: \mathcal{Y} \rightarrow \mathcal{X}$ of flat, integral, projective $\mathcal{O}_K$-schemes, we define similarly the pullback $f^\ast \Lotil \in\widehat{\Pic}_{\mathscr{C}^\infty}(\mathcal{Y})$ of any $\Lotil \in \widehat{\Pic}_{\mathscr{C}^\infty}(\mathcal{X})$.

The notion of vertical semipositivity introduced above descends to the isomorphism classes in $\widehat{\Pic}_{\mathscr{C}^\infty}(\mathcal{X})$. 
It further extends to $\widehat{\Pic}_{\mathscr{C}^\infty}(\mathcal{X})_\IQ = \widehat{\Pic}_{\mathscr{C}^\infty}(\mathcal{X}) \otimes_\IZ \IQ$. The elements of $\widehat{\Pic}_{\mathscr{C}^\infty}(\mathcal{X})_\IQ$ are called the $\mathscr{C}^\infty$-hermitian $\IQ$-line bundles on $\mathcal{X}$.

Every $\mathscr{C}^\infty$-hermitian line bundle $\overline{\mathcal{L}}$ on $\mathcal{X}$ gives rise to a metrized line bundle $\overline{\mathcal{L}}_K$ on $X$.  Such line bundles are called algebraically metrized (see \cite[Section 2.5]{Kuehne2022} for details). In addition, the metrized line bundle $\Lotil_K$ is vertically semipositive in the sense of \cite[Section 2.5]{Kuehne2022} if $\overline{\mathcal{L}}$ is so in the sense introduced above. (We do not claim the converse implication.) Furthermore, we note that $\Lotil_K$ is vertically integrable in the sense of \cite[Section 2.5]{Kuehne2022} for \textit{every} $\mathscr{C}^\infty$-hermitian line bundle $\overline{\mathcal{L}}$ on $\mathcal{X}$. In fact, we can always write $\overline{\mathcal{L}}$ as a difference $(\overline{\mathcal{L}} + \overline{\mathcal{O}}(k)|_{\mathcal{X}}) - \overline{\mathcal{O}}(k)|_{\mathcal{X}}$ with $\overline{\mathcal{O}}(k)$ being the $\mathscr{C}^\infty$-hermitian line bundle on $\IP^N_{\mathcal{O}_K}$ defined below. For sufficiently large integers $k$, both terms are semipositive because $\mathcal{O}(k)$ is ample and the metric on $\overline{\mathcal{O}}(k)$ is strictly positive.


Let us define the basic $\mathscr{C}^\infty$-hermitian line bundles that we use in this work. We consider the line bundle $\mathcal{O}(1)$ on $\IP^N_{\mathcal{O}_K}$ where $N$ is a positive integer. For each archimedean place $\nu \in \Sigma_\infty(K)$, we endow the holomorphic line bundle $\mathcal{O}(1)_{\IC_\nu}^\mathrm{an}$ on $\IP_{\IC_\nu}^N$ with the  Fubini-Study metric $\Vert \cdot \Vert_{\mathrm{FS},\nu}$ (see \cite[Subsection 3.3.2]{Voisin2007} for a definition). The Fubini-Study metric is smooth, so that we obtain a $\mathscr{C}^\infty$-hermitian line bundle $\overline{\caO}(1) = (\caO(1), \{ \Vert \cdot \Vert_{\mathrm{FS},\nu}\}_\nu) \in \widehat{\Pic}_{\mathscr{C}^\infty}(\IP^N_{\caO_K})$. For every integer $k$, we set $\overline{\caO}(k) := k \cdot \overline{\caO}(1)$. We write $\overline{\mathcal{O}}$ instead of $\overline{\mathcal{O}}(0)$. As the Chern form $\omega_{\mathrm{FS}}$ of the Fubini-Study metric is a strictly positive $\mathscr{C}^\infty$-hermitian form, each $\overline{\caO}(k)$ is vertically semipositive if $k \geq 1$. Note that we suppress the dimension $N$ here, as it will be always clear from context. On a biprojective space $\IP^{N_1}_{\caO_K} \times \IP^{N_2}_{\caO_K}$, we consider the $\mathscr{C}^\infty$-hermitian line bundles $\overline{\caO}(k_1,k_2)=\pr_1^\ast \overline{\caO}(k_1) + \pr^\ast_2 \overline{\caO}(k_2) \in \widehat{\Pic}_{\mathscr{C}^\infty}(\IP_{\caO_K}^{N_1} \times \IP_{\caO_K}^{N_2})$ where $\pr_i: \IP^{N_1}_{\caO_K}\times \IP^{N_2}_{\caO_K} \rightarrow \IP^{N_i}_{\caO_K}$, $i \in \{1,2\}$, is the projection to the $i$-th factor and $k_1,k_2$ are arbitrary integers. 

Consider next the $\mathscr{C}^\infty$-hermitian line bundle $(\mathcal{O}_\mathcal{X}, \{ \vert \cdot \vert_\nu \})$ where the hermitian metrics $|\cdot|_\nu$, $\nu \in \Sigma_\infty(K)$, are the trivial metrics. For a place $\nu \in \Sigma_\infty(K)$ and a $\Gal(\IC_{\nu}/K_{\nu})$-invariant function $f \in \mathscr{C}^\infty(X^{\mathrm{an}}_{\IC_{\nu}})$, we define furthermore $\overline{\mathcal{O}}(f) = (\mathcal{O}_{\mathcal{X}}, \{ \Vert \cdot \Vert_\nu \})$ by setting $\Vert \cdot \Vert_{\nu} = e^{-f}\vert \cdot \vert_{\nu}$ and $\Vert \cdot \Vert_\mu = \vert \cdot \vert_\mu$ for all places $\mu \in \Sigma_\infty(K) \setminus \{ \nu \} $. For a general $\Lotil \in \widehat{\Pic}_{\mathscr{C}^\infty}(\mathcal{X})$, we set $\Lotil(f) = \Lotil + \overline{\mathcal{O}}(f)$. 

\subsection{Arithmetic intersection numbers}
\label{subsection::arithmeticintersectionnumbers}
Given $\Lotil_1,\dots,\Lotil_{d^\prime+1} \in \widehat{\Pic}_{\mathscr{C}^\infty}(\mathcal{X})$, an arithmetic intersection number
\begin{equation}
\label{equation:intersection_number}
\Lotil_{1} \cdot \Lotil_{2} \cdots \Lotil_{d^\prime+1} \cdot \mathfrak{Z} \in \IR
\end{equation}
is defined in \cite[Section 2.6]{Kuehne2022} for every $d^\prime$-cycle $\mathfrak{Z}$ on $X$. 
To simplify our notation, we write just $$\Lotil_1 \cdot \Lotil_2 \cdots \Lotil_{d+1}$$ if additionally $d=d^\prime$ and $\mathfrak{Z}=[X]$. Note that the definition of \eqref{equation:intersection_number} in \cite{Kuehne2022} relies on Gubler's theory of local heights \cite{Gubler1997,Gubler1998,Gubler2003}. As we only work with  ${\mathscr{C}^\infty}$-hermitian line bundles on arithmetic varieties, the ``global'' approach as in \cite{Bost1994, Gillet1990, Soule1992} would be sufficient as well, and both approaches yield the same intersection numbers (\cite[Remark 11.24]{Gubler2003}).

For convenience of the reader, we briefly recall some properties of arithmetic intersection numbers as a lemma.

\begin{lemma}
	\label{lemma::intersectionnumber}
	Let $\Lotil_1,\dots,\Lotil_{d+1},\Lotil_1^\prime \in \widehat{\Pic}_{\mathscr{C}^\infty}(\mathcal{X})$ be vertically integrable. 
	\begin{enumerate}
		\item[(a)] (Multilinearity) We have
		\begin{equation*}
		(\Lotil_1 + \Lotil_1^\prime) \cdot \Lotil_2 \cdots \Lotil_{d+1} =
		\Lotil_1 \cdot \Lotil_2 \cdots \Lotil_{d+1} +\Lotil_1^\prime \cdot \Lotil_2 \cdots \Lotil_{d+1}.
		\end{equation*}
		\item[(b)] (Commutativity) For any permutation $\sigma: \{ 1, \dots, d + 1 \} \rightarrow \{1, \dots, d + 1 \}$,
		\begin{equation*}
		\Lotil_{\sigma(1)} \cdot \Lotil_{\sigma(2)} \cdots \Lotil_{\sigma(d+1)} = \Lotil_1 \cdot \Lotil_2 \cdots \Lotil_{d+1}.
		\end{equation*}
	\end{enumerate}
\end{lemma}

\begin{proof} These are just the first two assertions from \cite[Lemma 2.3]{Kuehne2022}.
\end{proof}

We conclude with a further observation: Let $\nu \in \Sigma_\infty(K)$ be an archimedean prime and let $f \in \mathscr{C}^\infty(X^{\mathrm{an}}_{\IC_{\nu}})$. For all $i \in \{1,\dots, d+1\}$ and $$\Lotil_1=(\mathcal{L}_1, \{ \Vert \cdot \Vert_{1,\nu}\}), \ \dots,\ \Lotil_{d+1-i}=(\mathcal{L}_{d+1-i}, \{ \Vert \cdot \Vert_{d+1-i,\nu}\}) \in \widehat{\Pic}_{\mathscr{C}^\infty}(\mathcal{X}),$$ we have 
\begin{multline}
\label{equation::intersection_integral}
\Lotil_1 \cdot \Lotil_2 \cdots \Lotil_{d+1-i} \cdot \overline{\caO}_{\mathcal{X}}(f)^i \\
= \delta_{\nu} \int_{X^{\mathrm{an}}_{\IC_{\nu}}} f (dd^c f)^{i-1} \wedge c_1(\Vert \cdot \Vert_{1,\nu}) \wedge c_1(\Vert \cdot \Vert_{2,\nu}) \wedge \dots \wedge c_1(\Vert \cdot \Vert_{d+1-i,\nu})
\end{multline}
by \cite[Equation (2.3) in Subsection 2.6]{Kuehne2022}\footnote{Note that this equation contains a typo on its left-hand side, which should be read as $\Lotil_1 \cdot \Lotil_2 \cdots \Lotil_{d+1-i} \cdot \overline{\caO}_{\mathcal{X}}(f)^i$.}.

Lemma \ref{lemma::intersectionnumber} (a) allows to linearly extend the definition of the arithmetic intersection number to all elements of $\widehat{\Pic}_{\mathscr{C}^\infty}(\mathcal{X})_\IQ$. 

\subsection{Heights} 
\label{subsection::heights} 
Let $\overline{\mathcal{L}}=(\mathcal{L}, \{ \Vert \cdot \Vert_\nu \}_{\nu \in \Sigma_\infty(K)}) \in \widehat{\Pic}_{\mathscr{C}^\infty}(\mathcal{X})$ be a $\mathscr{C}^\infty$-hermitian line bundle on $\mathcal{X}$. Assume that the generic fiber $L = \mathcal{L}_K$ is ample. For each irreducible subvariety $Y \subset X$ of dimension $d^\prime$, its Zariski closure $\mathcal Y$ in $\mathcal{X}$ is a flat, irreducible $\mathcal{O}_K$-scheme of relative dimension $d^\prime$ (see e.g.\ \cite[Proposition 4.3.9 and Corollary 4.3.14]{Liu2002}). Using the arithmetic intersection numbers introduced in the last subsection, we define the height 
\begin{equation*}
h_{\Lotil}(Y) = \frac{(\Lotil|_{\overline{\mathcal Y}})^{d^\prime+1}}{[K:\IQ](d^\prime+1)\deg_{L}(Y)}.
\end{equation*}

We can make this definition more explicit if $Y$ is a closed point $x \in X$. Writing $\overline{x}$ for the closure of $x$ in $\mathcal{X}$, the above definition simplifies to
\begin{equation*}
h_{\Lotil}(x) = \frac{\Lotil|_{\overline{x}}}{[K(x):\IQ]}
\end{equation*}
in this case. For every non-zero rational section $\mathbf{s}$ of $\mathcal{L}$ such that $\overline{x} \cap \left\vert \Div(\mathbf{s}) \right\vert = \emptyset$, this equals
\begin{equation}
\label{equation::height_section}
h_{\Lotil}(x) = \frac{1}{[K(x):\IQ]} \left( \log \# (\mathcal{L}|_{\overline{x}}/(\mathbf{s}|_{\overline{x}})) - \sum_{\nu \in \Sigma_\infty(K)} \sum_{y \in \mathbf{O}_\nu(x)}
\delta_\nu \log \Vert \mathbf{s}(y) \Vert_{\nu} \right).
\end{equation}

\subsection{Arithmetic volumes and Minkowski's Theorem}
\label{section::arithmeticvolumes}
Let $$\Lotil = (\mathcal{L}, {\{ \Vert \cdot \Vert_\nu } \}_{\nu \in \Sigma_\infty(K)})$$ be a $\mathcal{C}^\infty$-hermitian line bundle on $\mathcal{X}$. For each integer $N\geq 0$, we consider the isomorphism $$V_{N,\IR}=H^0(\mathcal{X},\mathcal{L}^{\otimes N}) \otimes_\IZ \IR \longrightarrow \prod_{\nu \in \Sigma_\infty(K)} H^0(\mathcal{X}_{\IC_\nu},\mathcal{L}_{\IC_\nu}^{\otimes N}).$$
Indeed, this generalizes a well-known isomorphism from Minkowski's geometry of numbers (compare \cite[Section I.5]{Neukirch1999}).\footnote{Note that we only have one factor for each complex archimedean place on the right-hand side here, so there is no action via complex conjugation in contrast to \cite{Neukirch1999}.} 
For each $\nu \in \Sigma_\infty(K)$, we can additionally endow $H^0(\mathcal{X}_{\IC_\nu},\mathcal{L}_{\IC_\nu}^{\otimes N})$ with a sup-norm by setting
\begin{equation*}
\Vert \mathbf{s} \Vert_{\nu}^{(\infty)} = \max_{x \in X_{\IC_\nu}^{\mathrm{an}}} \{ \Vert \mathbf{s}(x) \Vert_\nu^{\otimes N} \}
\end{equation*}
for every $\mathbf{s} \in H^{0}(\mathcal{X},\mathcal{L}^{\otimes N})$.
By restriction, we obtain a seminorm\footnote{This is indeed only a seminorm as the following example, suggested by the referee shows. If $K=\IQ(\sqrt{2})$, $\mathcal{X} = \caO_K$, $\mathcal{L}= \caO_{\mathcal{X}}$, then $\Vert \sqrt{2} \otimes 1 - 1 \otimes \sqrt{2} \Vert_{\nu}^\infty = \Vert 0 \Vert_\nu^{\infty} = 0$ for one archimedean place $\nu \in \Sigma_\infty(K)$.} on $V_{N,\IR}$. With these seminorms, we can define the unit ball
\begin{equation*}
B_N = \{ \mathbf{s} \in V_{N,\IR} \ | \ \forall \nu \in \Sigma_\infty(K):\Vert\mathbf{s}\Vert_{\nu}^{(\infty)} \leq 1 \}.
\end{equation*}
We let $\mathrm{vol}_N(\cdot)$ be the unique Haar measure on $V_{N,\IR}$ such that the induced quotient measure on $V_{N,\IR}/V_{N,\IZ}$ has total mass $1$. The arithmetic volume of $\Lotil \in \widehat{\Pic}_{\mathscr{C}^\infty}(\mathcal{X})$ is defined by
\begin{equation*}
\widehat{\vol}_{\chi}(\Lotil) = \limsup_{N \rightarrow \infty} \frac{\log \vol_N(B_N)}{N^{d+1}/(d+1)!}.
\end{equation*}
By \cite[Corollary 3.2.2]{Ikoma2013}, we have $\widehat{\vol}_{\chi}(\Lotil)<\infty$.

Arithmetic volumes enter our proof of Theorem \ref{theorem:equidistribution} through the construction of global sections with small supremum norms via Minkowski's Theorem. We summarize this in the following lemma.

\begin{lemma} 
	\label{lemma::minkowski}
	Let $\nu \in \Sigma_\infty(K)$ and a real $\varepsilon>0$ be given. Assume that the generic fiber $L = \mathcal{L}_K$ is nef and big. Then, there exists an integer $N_0$ and a non-zero section $\mathbf{s} \in H^0(\mathcal{X}, \mathcal{L}^{\otimes N_0})$ such that
	\begin{equation*}
	\delta_\nu \log \Vert \mathbf{s} \Vert_{\nu}^{(\infty)} \leq \left( - \frac{\volh_\chi(\Lotil)}{(d+1)L^d} + \varepsilon\right) N_0
	\end{equation*}
	and
	\begin{equation*}
	\log \Vert \mathbf{s} \Vert_\mu^{(\infty)}  \leq 0
	\end{equation*}
	for all other $\mu \in \Sigma_\infty(K) \setminus \{\nu \}$.
\end{lemma}

\begin{proof}
	This is a well-known consequence of Minkowski's Second Theorem \cite[Theorem C.2.11]{Bombieri2006}. See e.g.\ \cite[Lemma 2.7]{Kuehne2022} for a deduction in the slightly more general setting of adelically metrized line bundles.
\end{proof}

\section{The Equilibrium measure}
\label{section:equilibrium}

In this section, we describe -- up to a proportionality factor appearing in Lemma \ref{lemma::proportionality_constant} below -- the equilibrium measure $\mu_\nu$ postulated in Theorem \ref{theorem:equidistribution}, establishing also some essential lemmas for its proof in the next section. In contrast to the rest of the article, we work here with a smooth, irreducible \textit{complex} variety $S$, a family $\pi: A \rightarrow S$ of \textit{complex} abelian varieties of relative dimension $g$, an irreducible subvariety $X \subseteq A$ of dimension $d$ such that $\pi(X)=S$, and an immersion $\iota: A \hookrightarrow \IP^N_\IC$. Notationally, we identify the complex-analytic spaces associated with complex algebraic varieties with their sets of complex-valued points.

\textit{We assume throughout this section that the line bundle $\iota^\ast \caO(1)$ is fiberwise symmetric (i.e., $\iota^\ast \caO(1)|_s$ is a symmetric line bundle on $A|_s$ for every $s \in S(\IC)$). In addition, we let 
\begin{equation*}
	\boldsymbol{\Delta}=(\delta_1,\delta_2,\dots,\delta_g)\in \IZ^g, \ \delta_i \mid \delta_{i+1} \ (1 \leq i < g),
\end{equation*}	
denote the type of the polarization $\lambda: A \rightarrow A^\vee$ induced by $\iota^\ast \caO(1)$.}

\textit{All constants, whether implicit or explicit, are allowed to depend on the data introduced so far without further mention.}

Let $\omega_{\mathrm{FS}}$ denote the Fubini-Study form on $\IP^N(\IC)$. For each integer $k\geq 0$, we consider the $(1,1)$-form
\begin{equation*}
\alpha_k = \frac{(\iota \circ [n^k])^\ast \omega_{\mathrm{FS}}}{n^{2k}}
\end{equation*}
on $A(\IC)$. Note that $\alpha_{k+l} = ([n^l]^\ast \alpha_k) / n^{2l}$ for all integers $k,l \geq 0$.

\begin{lemma}
\label{lemma:equilibrium}
There exists a semipositive, smooth, closed $(1,1)$-form $\beta$ on $A(\IC)$ such that, for every irreducible subvariety $X \subseteq A$ of dimension $d$ and every compactly supported continuous function $f:X(\IC) \rightarrow \IR$, we have
\begin{equation}
\label{equation:equilibrium_lemma}
\lim_{k \rightarrow \infty}\int_{X(\IC)} f \alpha^{\wedge d}_k = \int_{X(\IC)} f \beta^{\wedge d}.
\end{equation}
Furthermore, $[n]^\ast \beta = n^2 \cdot \beta$.
\end{lemma}

In the next section, we prove that, for a non-degenerate subvariety $X$, the measure $\beta^{\wedge d}$ on $X(\IC)$ is indeed proportional to the equilibrium measure $\mu_\nu$ in Theorem \ref{theorem:equidistribution}.

\begin{proof}
To ease notation, we write $L$ instead of $\iota^\ast \caO(1)$ and $\Vert \cdot \Vert: L(\IC) \rightarrow \IR^{\geq 0}$ for the hermitian metric induced by the Fubini-Study metric on $\caO(1)$. 

By the theorem of the cube (\cite[Corollary 3 on p.\ 59]{Mumford1970}), we have $L|_{\pi^{-1}(s)}^{\otimes n^2} \approx [n]^\ast L|_{\pi^{-1}(s)}$ for every point $s \in S(\IC)$. A sufficiently general version of the seesaw theorem (\cite[Theorem 24.66]{Goertz2023}) 
implies that $L^{\otimes n^2} \otimes [n]^\ast L^{\otimes -1}$ is the pullback of a line bundle on $S$. Hence, every point $s_0 \in S(\IC)$ has a Zariski-open neighborhood $U \subseteq S$ such that there exists an isomorphism $\psi_1 : L|_{\pi^{-1}(U)}^{\otimes n^2} \rightarrow [n]^\ast L|_{\pi^{-1}(U)}$. For each integer $k \geq 0$, this induces further an isomorphism $\psi_k: L|_{\pi^{-1}(U)}^{\otimes n^{2k}} \rightarrow [n^k]^\ast L|_{\pi^{-1}(U)}$. We obtain a smooth hermitian metric $\Vert \cdot \Vert_k: L(\IC)|_{\pi^{-1}(U)} \rightarrow \IR^{\geq 0}$ by demanding 
\begin{equation}
\label{equation::homogen}
\Vert \mathbf{s}\Vert_k^{n^{2k}} = [n^k]^\ast \Vert \psi_k \circ \mathbf{s}^{\otimes n^{2k}} \Vert
\end{equation}
for every meromorphic section $\mathbf{s}$ of $L(\IC)|_{\pi^{-1}(U)}$ where
\begin{equation*}
	[n^k]^\ast \Vert \cdot \Vert: [n^k]^\ast L(\IC) \longrightarrow \IR^{\geq 0}
\end{equation*}
denotes the pullback of the metric $\Vert \cdot \Vert$ on $L(\IC)$ along $[n^k]$. Let $K \subseteq U(\IC)$ be a compact neighborhood of $s_0$. The function $\Vert \mathbf{s} \Vert_2 / \Vert \mathbf{s} \Vert_1: \pi^{-1}(K) \rightarrow \IR^{> 0}$ is independent of the chosen meromorphic section $\mathbf{s}$. Since it is continuous, it is uniformly bounded. A version of Tate's limiting argument (compare the proof of \cite[Theorem 2.2]{Zhang1995a}) shows that there exists a (unique) hermitian metric $\Vert \cdot \Vert_\infty$ on $L(\IC)|_{\pi^{-1}(K)}$ such that $\Vert \mathbf{s} \Vert_k/\Vert \mathbf{s} \Vert_\infty \rightarrow 1$ uniformly on $\pi^{-1}(K)$ as $k \rightarrow \infty$.

Let $V \subset A(\IC)$ be an open set such that $\overline{V} \subset \pi^{-1}(K)$. By choosing $V$ small enough, we can assure that there exists a holomorphic section $\mathbf{s}$ of $L(\IC)$ over $V$ such that $(- \log \Vert \mathbf{s} \Vert_k)$ is uniformly bounded on $V$ for all $k \geq 1$. Then, we have
\begin{equation*}
dd^c (-\log \Vert \mathbf{s} \Vert_k) = n^{-2k} dd^c (-\log ([n^k]^\ast \Vert \psi_k(\mathbf{s}^{\otimes n^{2k}})\Vert)) = \alpha_k|_V.
\end{equation*}
Set $\beta|_V = dd^c (- \log \Vert \mathbf{s} \Vert_\infty)$. The uniform convergence $\Vert \mathbf{s} \Vert_k \rightarrow \Vert \mathbf{s} \Vert_\infty$ implies by \cite[Corollary 1.6]{Demailly1993}\footnote{This statement also requires that $\log (- \Vert \mathbf{s} \Vert_{\infty})$ is plurisubharmonic. However, this is automatic by \cite[Theorem 2.9.14 (iii)]{Klimek1991}. 
Furthermore, the cited result is stated only for complex manifolds so we have to apply it for $V$ instead of $X(\IC)$. By \cite{Lelong1957}, integration along $V \cap X(\IC)$ defines a closed positive current $T$ on $V$. Therefore, \cite[Corollary 1.6]{Demailly1993} implies the weak convergence
\begin{equation*}
	dd^c (- \log \Vert \mathbf{s} \Vert_k)^{\wedge d} \wedge T \longrightarrow dd^c (- \log \Vert \mathbf{s} \Vert_\infty)^{\wedge d} \wedge T,
\end{equation*}
which yields \eqref{equation::integral_limits}.} that, for any continuous function $f: X(\IC) \rightarrow \IR$ with compact support in $V$, we have
\begin{equation}
\label{equation::integral_limits}
\int_{X(\IC)} f  \alpha_k^{\wedge d} \longrightarrow \int_{X(\IC)} f (\beta|_V)^{\wedge d}, \ k \rightarrow \infty.
\end{equation}

We claim that the currents $\beta|_V$ constructed in this 
way for varying $V$ glue to a $(1,1)$-current $\beta$ on $A$. For this purpose, it suffices to show that a different choice of $\psi_1$ gives rise to the same current $\beta|_V$. In fact, any other choice is of the form $\phi \cdot \psi_1$ where $\phi$ is a non-zero holomorphic function on $\pi^{-1}(U)$. By compactness, the function $\phi$ has to be constant on any fiber of $\pi$. The argument from \cite{Zhang1995a} shows that the quotient of the two respective limit metrics is $|\phi|^{1/n^2}$. Thus, the uniqueness of $\beta|_V$ follows from the calculation 
\begin{equation*}
dd^c \log |\phi|^{1/n^2} = \frac{1}{2n^2} \cdot \left( dd^c \log \phi + dd^c \log\overline{\phi}\right) = 0. 
\end{equation*}
A straightforward partition of unity argument allows to deduce \eqref{equation:equilibrium_lemma}. Furthermore, the homogeneity relation $[n]^\ast \beta = n^2 \cdot \beta$ follows from \eqref{equation::homogen}.

It remains to verify that $\beta$ is a smooth differential form. Each point $s_0 \in S(\IC)$ has a simply connected open neighborhood $U \subseteq S(\IC)$ such that there exists a holomorphic map $\mathcal{c}: U \rightarrow \mathcal{H}_g$ for which $A(\IC)|_U$ is the pullback of the family $\pi_{g,\boldsymbol{\Delta}}: \mathcal{A}_{g,\boldsymbol{\Delta}} \rightarrow \mathcal{H}_g$. Write $\varphi: A(\IC)|_U \rightarrow \mathcal{A}_{g,\boldsymbol{\Delta}}$ for the associated map. Recall that $L_{g,\boldsymbol{\Delta}}$ is the line bundle on $\mathcal{A}_{g,\boldsymbol{\Delta}}$ constructed in \cite[Section 8.7]{Birkenhake2004}. By \cite[Theorem 2.2.3]{Birkenhake2004}, the restrictions of $\varphi^\ast L_{g,\boldsymbol{\Delta}}^{\otimes 2}$ and $L(\IC)^{\otimes 2}$ are isomorphic as the associated Appell-Humbert data agree \cite[Corollary 2.3.7 and Lemma 8.7.1]{Birkenhake2004}. As above, we can shrink $U$ to assume that $\varphi^\ast L_{g,\boldsymbol{\Delta}}^{\otimes 2} \approx L(\IC)^{\otimes 2}|_U$. It is hence sufficient to show that $L_{g,\boldsymbol{\Delta}}$ can be endowed with a smooth hermitian metric $\Vert \cdot \Vert_{g,\boldsymbol{\Delta}}$ such that 
\begin{equation}
\label{equation::hermitian_isom}
(L_{g,\boldsymbol{\Delta}}, \Vert \cdot \Vert_{g,\boldsymbol{\Delta}})^{\otimes n^2} \approx [n]^\ast (L_{g,\boldsymbol{\Delta}}, \Vert \cdot \Vert_{g,\boldsymbol{\Delta}})
\end{equation}
as hermitian line bundles. In fact, this relation characterizes $\Vert \cdot \Vert_{g,\boldsymbol{\Delta}}$ up to a positive constant by \cite[Theorem 2.2 (b)]{Zhang1995a}, so that it produces the same Chern form as the metric $\Vert \cdot \Vert_\infty$ constructed above. Endow the trivial line bundle $\mathcal{O}_{\IC^g \times \mathcal{H}_g}$ with the smooth hermitian metric $\Vert \cdot \Vert_0$ given by setting
\begin{equation*}
\Vert f \Vert_0 = e^{-\pi \mathrm{Im}(\boldsymbol{z})^t \mathrm{Im}(\boldsymbol{\tau})^{-1} \mathrm{Im}(\boldsymbol{z})}|f|
\end{equation*}
for every function $f$ defined on an arbitrary open of $\IC^g \times \mathcal{H}_g$. From the explicit automorphy factors describing $L_{g,\boldsymbol{\Delta}}$ (see \cite[Section 8.7]{Birkenhake2004}), it is easy to read off that $\Vert \cdot \Vert_0$ descends to a metric $\Vert \cdot \Vert_{g,\boldsymbol{\Delta}}$ on $L_{g,\boldsymbol{\Delta}}$. Furthermore, the isomorphism $(\mathcal{O}_{\IC^g \times \mathcal{H}_g}, \Vert \cdot \Vert_0)^{\otimes n^2} = [n]^\ast (\mathcal{O}_{\IC^g \times \mathcal{H}_g},\Vert \cdot \Vert_0)$ descends to \eqref{equation::hermitian_isom}. 
\end{proof}


The next lemma gives an alternative characterization of the degeneracy locus of $X$ in terms of the $(1,1)$-form $\beta$ just constructed. For this purpose, we recall the notion of Betti rank for smooth points $x \in X^{\mathrm{sm}}(\IC)$ of an irreducible subvariety $X \subseteq A$: Choose an open subset $\pi(x)\in U \subseteq S(\IC)$ such that there exists a real-analytic isomorphism
\begin{equation}
\label{equation::real_analytic_trivialization}
a: A(\IC)|_U \longrightarrow (\IR/\IZ)^{2g} \times U
\end{equation}
that restricts to a group homomorphism on each fiber over $U$ and set 
\begin{equation}
	\label{equation::real_analytic_trivialization2}
	b = \pr_1 \circ a: A(\IC)|_U \longrightarrow (\IR/\IZ)^{2g}.
\end{equation}
We define the Betti rank of $X$ at $x$ as the rank of the real-analytic map $b$ at $x$, to wit
\begin{equation*}
	\mathrm{rank}_{\mathrm{Betti}}(X,x) = \dim_\IR(db(T_{\IR,x}X(\IC))).
\end{equation*}
It is easy to see that $\mathrm{rank}_{\mathrm{Betti}}(X,x)$ depends neither on the choice of $U$ nor $a$.

\begin{lemma}
	\label{lemma::finallemma}
	For each $x \in X^{\mathrm{sm}}(\IC)$, we have $\mathrm{rank}_{\mathrm{Betti}}(X,x) = 2 \dim(X)$ if and only if $(\beta|_X^{\wedge \dim(X)})_x \neq 0$.
\end{lemma}

This lemma renders the statement of \cite[Proposition 2.7]{Dimitrov2021a} more precise, which is crucial for some of our arguments below.

\begin{proof}
	Let $U \subseteq S(\IC)$ be an open subset containing $\pi(x)$ such that we have an isomorphism as in \eqref{equation::real_analytic_trivialization}. There exists an open subset $V \subset A(\IC)$ containing $x$ and a holomorphic map $\mathcal{l}: V \rightarrow \IC^g \times \mathcal{H}_g$ such that $2\cdot \beta|_{V} = \mathcal{l}^\ast \omega$ where $\omega$ is the $(1,1)$-form
	\begin{equation*}
		\omega = i \left(d\boldsymbol{z} - d\boldsymbol{\tau} \cdot \mathrm{Im}(\boldsymbol{\tau})^{-1} \cdot \mathrm{Im}(\boldsymbol{z}) \right)^t \wedge \mathrm{Im}(\boldsymbol{\tau})^{-1} \left(d\overline{\boldsymbol{z}} - d\overline{\boldsymbol{\tau}} \cdot \mathrm{Im}(\boldsymbol{\tau})^{-1} \cdot \mathrm{Im}(\boldsymbol{z}) \right) 
	\end{equation*}
	and $\boldsymbol{z}=(z_1,\dots,z_g)^t$ and $\boldsymbol{\tau} = (\tau_{ij})_{1\leq i,j \leq g}$ are the standard coordinates on $\IC^g \times \mathcal{H}_g$ (see \cite[Lemma 2.3]{Dimitrov2021a} and the proof of Lemma \ref{lemma:equilibrium} above). We can additionally arrange $\mathcal{l}$ such that the restriction $b|_V: A(\IC)|_V \rightarrow (\IR/\IZ)^{2g}$ from \eqref{equation::real_analytic_trivialization2} lifts to a map $\boldsymbol{b} = (b_1,\dots,b_{2g})^t: A(\IC)|_V \rightarrow \IR^{2g}$ such that 
	\begin{equation}
		\label{equation::finallemma_1}
		\boldsymbol{z} = (\boldsymbol{\tau},\mathrm{diag}(\boldsymbol{\Delta})) \cdot \boldsymbol{b}, \ \mathrm{diag}(\boldsymbol{\Delta}) = 
		\begin{pmatrix} 
			\delta_1 &  &  \\
			& \ddots & \\
			&  & \delta_g
		\end{pmatrix},
	\end{equation}
	(compare \cite[Section 8.1]{Birkenhake2004} and \cite[Proposition B.2]{Dimitrov2021a}). We use $\boldsymbol{z}$ and $\boldsymbol{\tau}$ also for the respective functions induced on $A(\IC)|_V$ through pullback along $\mathcal{l}$, abusing notation slightly. Taking the imaginary part of \eqref{equation::finallemma_1}, we obtain
	\begin{equation*}
		\mathrm{Im}(z_j) = \sum_{k=1}^{g} \mathrm{Im}(\tau_{jk}) b_{k}.
	\end{equation*}
	In terms of matrices,
	\begin{equation*}
		\begin{pmatrix}
			b_{1} \\
			\cdots \\
			b_{g}
		\end{pmatrix}
		=
		\mathrm{Im}(\boldsymbol{\tau})^{-1} \cdot \mathrm{Im}(\boldsymbol{z}).
	\end{equation*}
	We infer that
	\begin{equation*}
		2 \cdot \beta|_{V} = i \left(d\boldsymbol{z} - d\boldsymbol{\tau} \begin{pmatrix}
			b_{1} \\
			\cdots \\
			b_{g}
		\end{pmatrix} \right)^t \wedge \mathrm{Im}(\boldsymbol{\tau})^{-1} \left(d\overline{\boldsymbol{z}} - d\overline{\boldsymbol{\tau}} \begin{pmatrix}
			b_{1} \\
			\cdots \\
			b_{g}
		\end{pmatrix} \right). 
	\end{equation*}
	Choose a $\IC$-analytic chart
	\begin{equation*}
		\chi: B_1(0)^d = \{ (w_1,\dots,w_d) \in \IC^d \mid \max\{|w_1|,\dots,|w_d|\} < 1 \} \longrightarrow X^{\mathrm{sm}}(\IC) \cap V
	\end{equation*}	
	such that $\chi(0)=x$. For every function $f$ on $X(\IC)$, we simply write $f$ (resp.\ $\partial f/ \partial w_l$, $\partial f/ \partial \overline{w}_l$) instead of $f \circ \chi$ (resp.\ $\partial (f \circ \chi)/\partial w_l$, $\partial (f \circ \chi)/\partial \overline{w}_l$). Consider the $(g\times d)$-matrix
	\begin{equation*}
		\boldsymbol{J} = \left(\frac{\del z_j}{\del w_l} - \sum_{k=1}^{g} \frac{\del \tau_{jk}}{\del w_l} \cdot b_{k} \right)_{\substack{1\leq j \leq g \\ 1 \leq l \leq d}}.
	\end{equation*}
	of (complex-valued) real-analytic functions on $B_1(0)^d$. It is easy to compute that
	\begin{equation*}
		2 \cdot \chi^\ast \beta = i (\boldsymbol{J} \cdot d\boldsymbol{w})^t \wedge (\mathrm{Im}(\boldsymbol{\tau})^{-1} \overline{\boldsymbol{J}} \cdot d\overline{\boldsymbol{w}})
		= i (d\boldsymbol{w})^t \wedge (\boldsymbol{J}^t \mathrm{Im}(\boldsymbol{\tau})^{-1}\overline{\boldsymbol{J}} \cdot d\overline{\boldsymbol{w}}).
	\end{equation*}
	As $(\boldsymbol{J}^t \mathrm{Im}(\boldsymbol{\tau})^{-1} \overline{\boldsymbol{J}})(0) \in \IC^{d \times d}$ is a semipositive Hermitian matrix, there exists a unitary matrix $\boldsymbol{U} \in \IC^{d \times d}$ such that $\boldsymbol{J}^t \cdot \mathrm{Im}(\boldsymbol{\tau})^{-1} \cdot \overline{\boldsymbol{J}}= \overline{\boldsymbol{U}} \cdot \boldsymbol{D} \cdot \boldsymbol{U}^t$ with $\boldsymbol{D} \in \IR^{d \times d}$ a diagonal matrix having real entries 
	\begin{equation*}
		d_1 \geq d_2 \geq \cdots \geq d_r > d_{r+1} = \cdots = d_d = 0
	\end{equation*}
	(see e.g.\ \cite[Theorem 10.13]{Roman2005}). In local coordinates $v_1,\dots,v_d$ such that $\boldsymbol{w}=\boldsymbol{U} \cdot \boldsymbol{v}$, we have $d\boldsymbol{w}=\boldsymbol{U} \cdot d\boldsymbol{v}$ and hence
	\begin{equation*}
		2 \cdot (\chi^\ast \beta)_0 = i \sum_{j=1}^r d_j \cdot dv_j \wedge d\overline{v}_j.
	\end{equation*}
	We infer that $(\beta|_X^{\wedge d})_x \neq 0$ is equivalent to $r = d$. As $\mathrm{Im}(\boldsymbol{\tau})^{-1}$ is a (strictly) positive symmetric matrix, this is equivalent to $\boldsymbol{J}(0)$ having maximal rank $d$.
	
	To relate the matrix $\boldsymbol{J}(0) \in \IC^{g \times d}$ to the Betti rank, we notice that $\mathrm{rank}_{\mathrm{Betti}}(X,x)$ is the rank of the matrix
	\begin{equation*}
		\mathbf{B} = 
		\begin{pmatrix}
			\frac{\del b_1}{\del w_1} & \frac{\del b_1}{\del w_2} & \cdots & \frac{\del b_{1}}{\del w_d} & \frac{\del b_1}{\del \overline{w}_1} & \frac{\del b_1}{\del \overline{w}_2} & \cdots & \frac{\del b_1}{\del \overline{w}_{d}} \\
			\frac{\del b_2}{\del w_1} & \frac{\del b_2}{\del w_2} & \cdots & \frac{\del b_{2}}{\del w_d} & \frac{\del b_2}{\del \overline{w}_1} & \frac{\del b_2}{\del \overline{w}_2} & \cdots & \frac{\del b_2}{\del \overline{w}_{d}}  \\
			\cdots & \cdots & \cdots & \cdots & \cdots & \cdots & \cdots & \cdots \\
			\frac{\del b_{2g}}{\del w_1} & \frac{\del b_{2g}}{\del w_2} & \cdots & \frac{\del b_{2g}}{\del w_d} & \frac{\del b_{2g}}{\del \overline{w}_1} & \frac{\del b_{2g}}{\del \overline{w}_2} & \cdots & \frac{\del b_{2g}}{\del \overline{w}_{d}}
		\end{pmatrix}
		(0) \in \IC^{2g \times 2d}.
	\end{equation*}
	Taking derivatives in \eqref{equation::finallemma_1}, we obtain
	\begin{equation*}
		\frac{\del z_j}{\del w_l} = \sum_{k=1}^{g} \frac{\del \tau_{jk}}{\del w_l} \cdot b_k +  \sum_{k=1}^{g} \tau_{jk} \cdot \frac{\del b_k}{\del w_l} + \delta_j \cdot \frac{\del b_{g+j}}{\del w_l}
	\end{equation*}
	and
	\begin{equation*}
		0 = \frac{\del z_j}{\del \overline{w}_l} =
		\sum_{k=1}^{g} \tau_{jk} \cdot \frac{\del b_k}{\del \overline{w}_l} + \delta_j \cdot \frac{\del b_{g+j}}{\del \overline{w}_l}
	\end{equation*}
	for all $j \in \{1,\dots, g\}$ and $l \in \{1,\dots, d \}$.
	These equations imply that
	\begin{equation*}
		\begin{pmatrix}
			\boldsymbol{\tau}(0) & \mathrm{diag}(\boldsymbol{\Delta}) \\
			\overline{\boldsymbol{\tau}(0)} & \mathrm{diag}(\boldsymbol{\Delta})
		\end{pmatrix} \cdot \mathbf{B} 
		=
		\begin{pmatrix}
			\boldsymbol{J}(0) & \boldsymbol{0}_{g \times d} \\
			\boldsymbol{0}_{g \times d} & \overline{\boldsymbol{J}(0)}
		\end{pmatrix}.
	\end{equation*}
	As $\begin{pmatrix}
		\boldsymbol{\tau}(0) & \mathrm{diag}(\boldsymbol{\Delta}) \\
		\overline{\boldsymbol{\tau}(0)} & \mathrm{diag}(\boldsymbol{\Delta})
	\end{pmatrix} \in \IC^{2g \times 2g}$ is invertible, we conclude
	\begin{equation*}
		\mathrm{rank}_{\mathrm{Betti}}(X,x) = \mathrm{rank}(\mathbf{B}) = 2 \cdot \mathrm{rank}(\boldsymbol{J}(0)),
	\end{equation*}
	finishing the proof.
\end{proof}

We continue with establishing degree bounds for non-degenerate subvarieties. 
We let $\overline{X}_k \subset \IP^N_\IC$ denote the Zariski closure of $\iota([n^k](X))$. Furthermore, we consider the graph $\Gamma_k \subset A \times A$ of $[n^k]|_X: X \rightarrow [n^k](X)$ and the Zariski closure $\overline{Y}_k$ of $(\iota \times \iota)(\Gamma_k)$ in $\IP^N_\IC \times \IP^N_\IC$. We write $\caO(k_1,k_2)$ for the line bundle $\pr_1^\ast \caO(k_1) \otimes \pr_2^\ast \caO(k_2)$ on $\IP^N_\IC \times \IP^N_\IC$ where $\pr_i: \IP^{N}_{\IC}\times \IP^{N}_{\IC} \rightarrow \IP^{N}_{\IC}$, $i \in \{1,2\}$, is the projection to the $i$-th factor.

\begin{lemma}
\label{lemma::lowerdegreebound}
	If $X$ is non-degenerate, then $\deg_{\caO(1,1)}(\overline{Y}_k) \gg n^{2kd}$ for all integers $k \geq 1$.
\end{lemma}

\begin{proof}
Let $U \subseteq S(\IC)$ be a non-empty relatively compact open subset (i.e., a non-empty open subset whose closure in the euclidean topology is compact). As $X$ is non-degenerate, we have $\int_{\pi^{-1}(U) \cap X(\IC)} \beta^{\wedge d} > 0$ (compare \cite[Proposition 2.7]{Dimitrov2021a} or Lemma \ref{lemma::finallemma} below). By Lemma \ref{lemma:equilibrium} above, we deduce thus that 
\begin{equation*}
	\int_{\pi^{-1}(U) \cap X(\IC)} \alpha_k^{\wedge d} > c_{13}
\end{equation*}
for some constant $c_{13}>0$ and all integers $k\gg 1$. We recall that algebraic Chern classes and analytic Chern forms on proper complex algebraic varieties are compatible (i.e., yield the some intersection numbers) and refer to the paragraph before the proof of \cite[Lemma 29]{Kuehne2020} for details. This allows us to compute the degree of $\overline{Y}_k$ by integration. Using the semipositivity of $\alpha_0$ and $\alpha_k$, we obtain in this way that
\begin{align*}
\deg_{\caO(1,1)}(\overline{Y}_k) 
&= \int_{X(\IC)} \left(\alpha_0 + n^{2kd} \cdot \alpha_k \right)^{\wedge d} \\
&> n^{2kd} \int_{X(\IC)} \alpha_k^{\wedge d} \\
&\geq n^{2kd} \int_{\pi^{-1}(U) \cap X(\IC)} \alpha_k^{\wedge d} \\
&> c_{13} \cdot n^{2kd}. \qedhere
\end{align*}
\end{proof}

We also need a converse bound, whose proof is purely algebraic. 

\begin{lemma}
\label{lemma::upperdegreebound}
For each integer $k\geq 1$, we have
\begin{equation*}
\deg_{\caO(1,1)}(\overline{Y}_k) \leq \deg_{\caO(n^{2k},1)}(\overline{Y}_k) \ll n^{2kd}.
\end{equation*}
\end{lemma}

This is proven similar to \cite[Section 4.2]{Dimitrov2021a} and we refer to there for some parts of the proof that run parallel to our argument.

\begin{proof}
The first inequality is a direct consequence of Kleiman's criterion \cite[Theorem III.2.1]{Kleiman1966}, using the fact that $\caO(n^{2k}-1,0)$ is evidently nef. Therefore, we concentrate on proving the second one in the following. 

Recall that $\overline{X}$ (resp.\ $\overline{Y}_k$) is constructed as a closed subvariety of $\IP^N_\IC$ (resp.\ $\IP^N_\IC \times \IP^N_\IC$). Choosing a projective immersion $\kappa: S \hookrightarrow \IP^M_\IC$, we can consider $\overline{X}$ (resp.\ $\overline{Y}_k$) also as a subvariety of $\IP^N_\IC \times \IP^M_\IC$ (resp.\ $\IP^N_\IC \times \IP^M_\IC \times \IP^N_\IC$). We do so in the remainder of this proof. Write 
$\IC[\boldsymbol{Y}^{(1)},\boldsymbol{Y}^{(2)},\boldsymbol{Y}^{(3)}]$ for the multi-homogeneous coordinate ring of $\IP^N_\IC \times \IP^M_\IC \times \IP^N_\IC$. From the proof of the related \cite[Proposition 4.3]{Dimitrov2021a}, we extract the following fact: For each integer $k \geq 1$, there exist multi-homogeneous polynomials 
\begin{equation*}
	p_i(\boldsymbol{Y}^{(1)},\boldsymbol{Y}^{(2)},\boldsymbol{Y}^{(3)}) \in \IC[\boldsymbol{Y}^{(1)},\boldsymbol{Y}^{(2)},\boldsymbol{Y}^{(3)}], \ 1 \leq i \leq N,
\end{equation*}
with multi-degrees $(\delta_1^{(i)}, \delta_2^{(i)},1)$ such that $\delta_1^{(i)}, \delta_2^{(i)} \ll n^{2k}$ and such that $\overline{Y}_k$ is an irreducible component of
\begin{equation}
	\label{equation::intersection_polynomials}
	V(p_1) \cap \cdots \cap V(p_N) \cap (\overline{X} \times \IP^N_\IC) \subseteq \IP^N_\IC \times \IP^M_\IC \times \IP^N_\IC.
\end{equation}

We use this to bound the degree of $\overline{Y}_k$. For this purpose, we also use the basic notations and results from \cite{Fulton1998}. By \cite[Examples 1.9.3 and 8.3.7]{Fulton1998}, the Chow ring $A^{\ast}(\IP^{N}_\IC \times \IP^{M}_\IC \times \IP^N_\IC)$ is of the form
\begin{equation*}
 	\IZ[H_1]/([H_1]^{N+1}) \otimes \IZ[H_2]/([H_2]^{M+1}) \otimes \IZ[H_3]/([H_3]^{N+1})
\end{equation*}
where $H_i$, $1 \leq i \leq 3$, is the preimage of an arbitrary hyperplane along the $i$-th projection. Intersecting with generic hyperplanes, we infer that
\begin{equation}
	\label{equation::chow2}
	[V(p_i)] = \delta_1^{(i)}\cdot [H_1] + \delta_2^{(i)} \cdot [H_2] + [H_3], \ 1 \leq i \leq N,
\end{equation}
and that we can write
\begin{equation}
	\label{equation::chow1}
	[\overline{X} \times \IP^N_\IC] = \sum_{j+k=M+N-d} a_{j,k} \cdot ([H_1]^{j} \otimes [H_2]^{k} \otimes [H_3]^0).
\end{equation}
with non-negative integers $a_{j,k}$.

We claim next that there exists a chain of irreducible subvarieties 
\begin{equation*}
	W_i \subseteq \IP^N_\IC \times \IP^M_\IC \times \IP^N_\IC, \ 0 \leq i \leq N,
\end{equation*}
satisfying the following properties
\begin{enumerate}
	\item[(i)] $W_0 = \overline{X} \times \IP^N_\IC$,
	\item[(ii)] $W_{N} = \overline{Y}_k$,
	\item[(iii)] $\dim(W_i) = d + (N - i)$,
	\item[(iv)] there exist non-negative integers $a_{j,k,l}^{(i)} \ll_i n^{2(i-l)k}$ such that
	\begin{equation*}
		[W_i] = \sum_{j+k+l=M+N+i-d} a_{j,k,l}^{(i)} \cdot [H_1]^{j} \otimes [H_2]^{k} \otimes [H_3]^{l}.
	\end{equation*}
\end{enumerate}
Before proving the claim, let us show that it is sufficient to derive the second degree bound of the lemma. Indeed, $[\overline{Y}_k] = [W_N]$ implies that $\deg_{\caO(n^{2k},1)}(\overline{Y}_k)$ equals the degree of the $0$-cycle
\begin{equation*}
	(n^{2k} \cdot [H_1] + [H_3])^d \cdot \sum_{j+k+l = M+2N-d} a_{j,k,l}^{(N)} \cdot ([H_1]^{j} \otimes [H_2]^{k} \otimes [H_3]^l), 
\end{equation*}
which expands as
\begin{equation*}
	\left(\sum_{j=0}^d \binom{d}{j} n^{2jk} a_{N-j,M,N-d+j}^{(N)} \right) [H_1]^N \otimes [H_2]^M \otimes [H_3]^N.
\end{equation*}
Hence, the bounds in (iv) for $i=N$ imply that
\begin{equation*}
	\deg_{\caO(n^{2k},1)}(\overline{Y}_k) 
	= 
	\sum_{j=0}^d \binom{d}{j} n^{2jk} a_{N-j,M,N-d+j}^{(N)}
	\ll n^{2kd}.
\end{equation*}

Therefore, we finish with the proof of the above claim. Starting with $W_0 = \overline{X} \times \IP^N_\IC$, we choose the $W_i$ iteratively. Assuming that $W_i$, $0 \leq i < N$, has already been chosen, we let $W_{i+1}$ be the unique irreducible component of $W_i \cap V(p_{i+1})$ containing $\overline{Y}_k$. The intersection $W_i \cap V(p_{i+1})$ is non-empty because it contains  $\overline{Y}_k$. By Krull's principal ideal theorem, this means that either $W_i \cap V(p_{i+1}) = W_i$ and hence $W_i=W_{i+1}$ or $\dim(W_{i+1}) = \dim(W_i) - 1$. In particular, the dimension drops at most by $1$ in each step. As $W_N$ has to contain $\overline{Y}_k$ by construction and is contained itself in \eqref{equation::intersection_polynomials}, of which $\overline{Y}_k$ is an irreducible component, we also know that (ii) $W_N = \overline{Y}_k$. From $\dim(W_0)= d + N$ and the fact that $\overline{Y}_k = W_N$ has dimension $d$, we conclude that the dimension has to drop by $1$ in each step and (iii) holds. We prove (iv) inductively. The base case $i=0$ is just \eqref{equation::chow1}. So assume that we have already proven that
\begin{equation*}
	[W_i] = \sum_{j+k+l=M+N+i-d} a_{j,k,l}^{(i)} \cdot [H_1]^{j} \otimes [H_2]^{k} \otimes [H_3]^{l}
\end{equation*}
with integers $a_{j,k,l}^{(i)} \ll_i n^{2(i-l)k}$. We consider the intersection product $[W_{i}] \cdot [V(p_{i+1})]$ as defined in \cite[Chapter 8]{Fulton1998}. As $W_{i+1}$ is a proper component of the intersection, it is one of its distinguished varieties by \cite[Lemma 7.1]{Fulton1998}. By \cite[Proposition 7.1 (a)]{Fulton1998}, it contributes a positive multiple of its class $[W_{i+1}]$ to the intersection product. As the tangent bundle on multiprojective spaces is globally generated, each other distinguished variety contributes a non-negative cycle to the intersection product (\cite[Corollary 12.2 (a)]{Fulton1998}). In terms of Chow classes, this means that we can write
\begin{equation*}
	[W_i] \cdot [V(p_{i+1})] = [W_{i+1}] + E
\end{equation*}
with $E \in A^{M+N+i+1-d}(\IP^N_\IC \times \IP^M_\IC \times \IP^N_\IC)_{\geq 0}$ a non-negative Chow class. It is hence sufficient to express $[W_i] \cdot [V(p_{i+1})]$ with coefficients bounded as in (iv). A simple calculation in the Chow ring using \eqref{equation::chow2} yields that $[W_i] \cdot [V(p_{i+1})]$ equals
\begin{equation*}
	\sum_{j+k+l=M+N+i+1-d} b_{j,k,l}^{(i+1)} \cdot [H_1]^{j} \otimes [H_2]^{k} \otimes [H_3]^{l}
\end{equation*}
with 
\begin{equation*}
	b_{j,k,l}^{(i+1)}=\delta_1^{(i)} a_{j-1,k,l}^{(i)} + \delta_2^{(i)} a_{j,k-1,l}^{(i)} + a_{j,k,l-1}^{(i)} \ll_{i+1} n^{2(i+1-l)k}, 
\end{equation*}
which completes the proof.
\end{proof}

We complement Lemma \ref{lemma:equilibrium} with a further, rather trivial estimate.

\begin{lemma}
\label{lemma:rathertrivialestimate}	
	For every smooth, compactly supported $2i$-form $\gamma$ on $X(\IC)$, we have
	\label{lemma::fs_bound}
	\begin{equation*}
	\int_{X(\IC)} \alpha^{\wedge (d-i)}_k \wedge \gamma \ll_{\gamma} 1.
	\end{equation*}
\end{lemma}

\begin{proof}
	Let $U \subset S(\IC)$ be a (non-empty) open subset (in the euclidean topology) such that there exists a real-analytic isomorphism $a: A(\IC)|_{U} \rightarrow (\IR/\IZ)^{2g} \times U$ as in \eqref{equation::real_analytic_trivialization}. Let $u_1,\dots,u_{2g}$ be the pullback of the standard real-analytic coordinates on $(\IR/\IZ)^{2g}$ to $(\IR/\IZ)^{2g} \times U$. Similarly, we assume that $U$ is small enough such that there exists a system of local real-analytic coordinates on $U$ as well and we write $u_1^\prime,\dots,u_{2\mathcal{s}}^\prime$ for its pullback to $(\IR/\IZ)^{2g} \times U$ along the second projection. Using these real-analytic coordinates on $A(\IC)|_U$, we can write
	\begin{equation*}
	\alpha_0|_{\pi^{-1}(U)} = \sum_{1\leq i < j \leq 2g} g_{i,j} du_i \wedge du_j + \sum_{1 \leq i \leq 2 \mathcal{s}} h_i du_i^\prime \wedge \gamma_i
	\end{equation*}
	for $1$-forms $\gamma_i$ and smooth complex-valued functions $g_{i,j},h_i$ on $A(\IC)|_{U}$. 
	
	By a partition of unity argument, we can assume that the support of $\gamma$ is a compact subset $K$ contained in $\pi^{-1}(U)$. By compactness, we have
	\begin{equation*}
	\sup_{x \in K} \left\{ \max \{|g_{i,j}(x)|, |h_i(x)| \} \right\} \ll_{K} 1.
	\end{equation*}
	As $[n]^\ast du_i = n \cdot du_i$ ($1 \leq i \leq 2g$) and $[n]^\ast du_i^\prime = du_i^\prime$ ($1 \leq i \leq 2\mathcal{s}$), we immediately see that the coefficients of the differential form 
	\begin{equation*}
	\alpha_k^{\wedge (d-i)} \wedge \gamma = n^{-2k(d-i)} ([n^{k}]^\ast \alpha_0^{\wedge (d-i)} \wedge  \gamma)
	\end{equation*}
	are $\ll_{K, \gamma} 1$ on $K$, and the assertion of the lemma follows.
\end{proof}

In conclusion of this section, we state four simple lemmas for use in Section \ref{section::uniformity}. For any variety $Y$ over $S$, we write $Y^{[n]}$ for the $n$-fold fiber product $Y \times_S \cdots \times_S Y$. In addition, we let $\pi^{[n]}$ denote the projection $A^{[n]} \rightarrow S$ and, for an immersion $\iota: A \hookrightarrow \IP^N_\IC$, we set
\begin{equation}
	\label{equation::iota_n}
	\iota^{[n]} = \sigma \circ (\iota \times \cdots \times \iota)|_{A^{[n]}}: A^{[n]} = A \times_S \cdots \times_S A \longhookrightarrow \IP^{N_n}_\IC
\end{equation}
where $\sigma: \IP^N_\IC \times \cdots \times \IP^N_\IC \hookrightarrow \IP^{N_n}_\IC$, $N_n = (N+1)^n-1$, is the Segre embedding.

\begin{lemma}
	\label{lemma::non_degeneracy}
	Let $X, Y \subset A$ be irreducible subvarieties such that $\pi(X)=\pi(Y) = S$. Assume that $X$ is non-degenerate. Then, there exists a non-empty Zariski-open set $U \subseteq S$ such that $X \times_U Y$ is a non-degenerate irreducible subvariety of $A \times_U A$.
\end{lemma}
\begin{proof}
	Write $\eta$ for the generic point of $S$. As $X$ and $Y$ are irreducible, so are $X_\eta$ and $Y_\eta$ (compare \cite[(0.2.1.8)]{EGA1}). Hence, the generic fiber $(X \times_S Y)_\eta$ is irreducible. Again by \textit{loc.}\ \textit{cit.}, there is a unique irreducible component $Z \subseteq X \times_S Y$ intersecting $(X \times_S Y)_\eta$. Hence, there exists an open dense subset $U \subseteq S$ such that $X \times_U Y$ is irreducible. To prove that $X \times_U Y$ is non-degenerate, we may and do assume $U=S$ in the sequel. Furthermore, we set $\mathcal{s}=\dim(S)$, $d = \dim(X)$, and $d^\prime = \dim(Y)$.

	
	Let $U \subseteq S(\IC)$ be a non-empty simply connected open. Then there exists a real-analytic isomorphism
	\begin{equation*}
		a: A(\IC)|_U \longrightarrow (\IR/\IZ)^{2g} \times U
	\end{equation*}
	that is fiberwise also a homomorphism of real Lie groups. Furthermore, set 
	\begin{equation*}
		b =\pr_1 \circ a: A(\IC)|_U \longrightarrow (\IR/\IZ)^{2g}.
	\end{equation*}	
	For each integer $n \geq 1$, we set
	\begin{equation*}
		b^{[n]} = b \times_S \cdots \times_S b : A^{[n]}(\IC)|_U \longrightarrow (\IR/\IZ)^{2gn}.
	\end{equation*}
	By generic smoothness, we may pick a smooth point $x \in X(\IC)$ (resp.\ $y \in Y(\IC)$) for the morphism $\pi|_{X}: X \rightarrow S$ (resp.\ $\pi|_Y: Y \rightarrow S$). This implies
	\begin{equation}
		\label{equation::tangent_basis}
		d\pi(T_{\IR,x} X(\IC)) = T_{\IR,s} S(\IC) = d\pi(T_{\IR,y} Y(\IC)).
	\end{equation}
	We may additionally assume that
	\begin{equation}
		\label{equation::rank_at_x}
		\dim_\IR(db (T_{\IR,x}X(\IC))) 
		= \mathrm{rank}_{\mathrm{Betti}}(X,x) = 2d
	\end{equation}
	as the points satisfying this condition are real-analytically dense in $X(\IC)$. Computing tangent space dimensions, one can check that $z := (x,y) \in {(X \times_S Y)^{\mathrm{sm}}(\IC)}$. We claim that
	\begin{align*}
		\mathrm{rank}_{\mathrm{Betti}}(Z,z) =
		\dim_\IR(db (T_{\IR,z}Z(\IC))) = 2\dim(Z) = 2(d + d^\prime - \mathcal{s})
	\end{align*}
	with $\mathcal{s} = \dim(S)$. For the proof, we complete a basis $t_1,\dots,t_{2(d - \mathcal{s})} \in T_{\IR,x} X_{s}(\IC)$ of the ``vertical'' tangent vectors to a basis $t_1,\dots,t_{2d}$ of the full tangent space $T_{\IR,x}X(\IC)$. By our choice of $x$, the vectors
	\begin{equation*}
		db(t_1),\dots,db(t_{2d}) \in T_{\IR,b(x)} (\IR/\IZ)^{2g} = \IR^{2g}
	\end{equation*}
	are $\IR$-linearly independent. Let furthermore $t^\prime_1,\dots,t^\prime_{2(d^\prime-\mathcal{s})} \in T_{\IR,y} Y_{s}(\IC)$ be a basis of the ``vertical'' tangent vectors of $Y$ at $y$. Using \eqref{equation::tangent_basis}, we can also pick tangent vectors $t^\dprime_{2(d - \mathcal{s})+1},\dots,t^\dprime_{2d} \in T_{\IR,y} Y(\IC)$ such that 
	\begin{equation*}
		d\pi(t_i) = d\pi(t_i^\dprime), \ 2(d - \mathcal{s}) + 1 \leq i \leq 2d.
	\end{equation*}
	We notice that the ``vertical'' tangent vectors
	\begin{equation}
		\label{equation::diagonal_vectors_1}
		(t_i,0) \in (T_{\IR,x}X \times_{T_{\IR,s}S} T_{\IR,y}Y)(\IC) = T_{\IR,z}Z(\IC), \ 1\leq i \leq 2(d - \mathcal{s}),
	\end{equation}
	and 
	\begin{equation}
		\label{equation::diagonal_vectors_2}
		(0,t_i^\prime) \in (T_{\IR,x}X \times_{T_{\IR,s}S} T_{\IR,y}Y)(\IC) = T_{\IR,z}Z(\IC), \ 1\leq i \leq 2(d^\prime - \mathcal{s}),
	\end{equation}
	form a basis of the space $T_{\IR,z}Z_s(\IC)$, and together with the vectors
	\begin{equation}
		\label{equation::diagonal_vectors_3}
		(t_i,t_i^\dprime) \in (T_{\IR,x}X \times_{T_{\IR,s}S} T_{\IR,y}Y)(\IC) = T_{\IR,z}Z(\IC), \ 2(d - \mathcal{s}) + 1 \leq i \leq 2d,
	\end{equation} 
	they form a basis of $T_{\IR,z}Z(\IC)$. As $b^{[2]} = b \times_S b$, it is easy to see that the images of the $2 (d + d^\prime - \mathcal{s})$ tangent vectors in \eqref{equation::diagonal_vectors_1}, \eqref{equation::diagonal_vectors_2} and \eqref{equation::diagonal_vectors_3} under $db^{[2]}$ are $\IR$-linearly independent. In fact, if
	\begin{equation*}
		\sum_{i=1}^{2(d-\mathcal{s})} a_i \cdot (db(t_i), 0) +
		\sum_{i=1}^{2(d^\prime-\mathcal{s})} a_i^\prime \cdot (0, db(t_i^\prime)) +
		\sum_{i=2(d - \mathcal{s}) + 1}^{2d} a_i \cdot (db(t_i), db(t_i^\dprime)) = 0
	\end{equation*}
	is an $\IR$-linear equation, then
	\begin{equation*}
		a_{1}\cdot db(t_1) + \cdots + a_{2d} \cdot db(t_{2d}) = 0.
	\end{equation*}
	By \eqref{equation::rank_at_x}, it follows that
	\begin{equation*}
		a_1 = \cdots = a_{2d} = 0
	\end{equation*}
	and thus also
	\begin{equation*}
		a_1^\prime = \cdots = a_{2(d^\prime - \mathcal{s})}^\prime = 0.
		\qedhere
	\end{equation*}
\end{proof}

Lemma \ref{lemma::non_degeneracy} has the following immediate consequence: If $X \subset A$ is non-degenerate, then so is any fibered power $X^{[n]} \subset A^{[n]}$, $n \geq 1$. The following lemma provides more detailed information.

\begin{lemma}
	\label{lemma::non_degeneracy2}	
	Let $n \geq 1$ be an arbitrary integer.	If
	\begin{equation*}
		\mathrm{rank}_{\mathrm{Betti}}(X,x) = 2 \dim(X)
	\end{equation*}
	for a smooth point $x \in X(\IC)$ of the restriction $\pi|_X: X \rightarrow S$, then the diagonal embedding $(x,\dots,x) \in A^{[n]}(\IC)$ of $x$ is a smooth point of $X^{[n]}$ such that	
	\begin{equation*}
		\mathrm{rank}_{\mathrm{Betti}}(Z,(x,\dots,x)) = 2 \dim(X^{[n]}) = 2 (n\dim(X) - (n-1)\dim(S))
	\end{equation*}
	where $Z \subseteq X^{[n]}$ is the unique irreducible component containing $(x,\dots,x)$.
\end{lemma}

\begin{proof}
	Computing dimensions of tangent spaces, we note first that $X^{[n]}$ is smooth at $(x,\dots,x)$. Hence, there is a unique irreducible component $Z \subseteq X^{[n]}$ containing $(x,\dots,x)$. 
	
	The remaining parts of the lemma follow from an argument similar to the one given for Lemma \ref{lemma::non_degeneracy} above. In fact, setting $y = x$ and $z = (x,x)$ in its proof, we can take $t_i=t^\dprime_i$ ($i \in \{2(d - \mathcal{s}) + 1,\dots, 2d\}$). This already yields the lemma in the case $n=2$. 
	
	In the general case, let $t_1,\dots,t_{2d}$ be a basis of the tangent space $T_{\IR,x}X(\IC)$ such that $t_1,\dots,t_{2(d - \mathcal{s})} \in T_{\IR,x} X_{s}(\IC)$ form a basis of the ``vertical'' tangent space. We set $z=(x,\dots,x)$. The ``vertical'' tangent vectors
	\begin{multline}
		\label{equation::diagonal_vectors_1b}
		(0,\dots,\overset{\text{$j$-th position}}{\overset{\downarrow}{t_i}},\dots,0) \in T_{\IR,z}X^{[n]}(\IC), \ 1\leq i \leq 2(d - \mathcal{s}), \ 1 \leq j \leq n,
	\end{multline}
	form a basis of the space $T_{\IR,z}X_{s}^{[n]}(\IC)$, and together with the vectors
	\begin{equation}
		\label{equation::diagonal_vectors_2b}
		(t_i,t_i,\dots,t_i) \in T_{\IR,z}X(\IC)^n, \ 2(d - \mathcal{s}) + 1 \leq i \leq 2d,
	\end{equation} 
	they form a basis of $T_{\IR,z}X^{[n]}(\IC)$. As $b^{[n]} = b \times_S \cdots \times_S b$, one can check as above that the images of the $2((d-\mathcal{s})n + \mathcal{s})$ tangent vectors in \eqref{equation::diagonal_vectors_1b} and \eqref{equation::diagonal_vectors_2b} under $db^{[n]}$ are $\IR$-linearly independent.
\end{proof}

Furthermore, we note that non-degeneracy is preserved under isogenies.

\begin{lemma}
	\label{lemma::nondegenerated_isogenies}
	Let $\pi^\prime: A^\prime \rightarrow S$ be another family of complex abelian varieties and $q: A \rightarrow A^\prime$ a fiberwise isogeny over $S$. If $X \subseteq A$ is a non-degenerate irreducible subvariety, then so is its image $q(X) \subseteq A^\prime$.
\end{lemma}

\begin{proof}
	Choose a simply connected open $U \subseteq S(\IC)$ as well as Betti maps 
	\begin{equation*}
		b: (A \cap \pi^{-1}(U))(\IC) \longrightarrow (\IR/\IZ)^{2g}
	\end{equation*}
	and
	\begin{equation*}
		b^\prime: (A^\prime \cap \pi^{-1}(U))(\IC) \longrightarrow (\IR/\IZ)^{2g}.
	\end{equation*}
	As $X$ is non-degenerate, we may pick a point $x_0 \in X(\IC)$ such that
	\begin{equation*}
		\dim_\IR(db(T_{\IR,x_0}X(\IC))) = \mathrm{rank}_{\mathrm{Betti}}(X,x_0) = 2\dim(X).
	\end{equation*}
	As the Betti map is compatible with the fiberwise group structures, there exists an isogeny $q^\prime: (\IR/\IZ)^{2g} \rightarrow (\IR/\IZ)^{2g}$ such that
	\begin{equation*}
		\begin{tikzcd}
			(A \cap \pi^{-1}(U))(\IC) \ar[r, "b"] \ar[d, "q"] & (\IR/\IZ)^{2g} \ar[d, "q^\prime"] \\
			(A^\prime \cap \pi^{-1}(U))(\IC) \ar[r, "b^\prime"] & (\IR/\IZ)^{2g}
		\end{tikzcd}
	\end{equation*}
	commutes. We infer that
	\begin{equation*}
		db^\prime (T_{\IR,q(x_0)}q(X)(\IC)) = db^\prime(dq(T_{\IR,x_0}X(\IC))) = 
		dq^\prime (db(T_{\IR,x_0}X(\IC))).
	\end{equation*}
	As $dq^\prime$ is an $\IR$-linear isomorphism, this implies
	\begin{equation*}
		\dim_\IR(db^\prime(T_{\IR,q(x_0)}q(X)(\IC))) = \dim_\IR(db(T_{\IR,x_0}X(\IC))) = 2\dim(X) = 2\dim(q(X)).
	\end{equation*}
	Hence $q(X)$ is non-degenerate.
\end{proof}

We conclude this section with establishing the non-proportionality of two volume forms derived from the Betti form. These forms appear naturally in the proof of Proposition \ref{proposition:bogo2} below. For each integer $n \geq 1$, we write $\beta^{[n]}$ for the $(1,1)$-form defined for $\pi^{[n]}: A^{[n]} \rightarrow S$ in Lemma \ref{lemma:equilibrium}. For a given integer $m \geq 2$, we define the map	
\begin{align*}
	\Delta_{0}: \ \ &A^{[m]}&  &\longrightarrow& & A^{[m-1]}, \\ &(x_1,x_2,\dots,x_{m}) & &\longmapsto& &(x_1-x_2,x_2-x_3, \dots,x_{m-1}-x_m),
\end{align*}
and set
\begin{equation*}
	\Delta = \Delta_{0} \times_S \id_{A}: \ A^{[m]} \times_S A  \longrightarrow A^{[m-1]} \times_S A.
\end{equation*}
We call two volumes forms $\alpha_1, \alpha_2$ proportional if there exists a real number $c \neq 0$ such that $\alpha_1 = c \cdot \alpha_2$.

\begin{lemma}
	\label{lemma::non_proportional}
	Let $m \geq 2$ be an arbitrary integer and $X \subseteq A$ be an irreducible subvariety of dimension $d>\dim(S)$ with $\pi(X)=S$. If $X$ is non-degenerate, then the differential forms
	\begin{equation*}
	(\beta^{[m+1]})^{\wedge \dim(X^{[m+1]})} \ \ \text{and} \ \	(\Delta^\ast \beta^{[m]})^{\wedge \dim(X^{[m+1]})}
	\end{equation*}
	do not restrict to proportional volume forms on $X^{[m+1]}$.
\end{lemma}

\begin{proof}
As $X$ is non-degenerate, there exists a smooth point $x \in X$ of the morphism $\pi|_{X}: X \rightarrow S$ such that $\mathrm{rank}_{\mathrm{Betti}}(X, x) = 2\dim(X)$. Using Lemma \ref{lemma::non_degeneracy2}, we infer that
\begin{equation*}
	\mathrm{rank}_{\mathrm{Betti}}(Z, p) = 2 \dim(X^{[m+1]})
\end{equation*}
for the diagonally embedded point $p=(x,x,\dots,x)\in X^{[(m+1)],\mathrm{sm}}(\IC)$ and the unique irreducible component $Z \subseteq X^{[m+1]}$ containing $p$. Lemma \ref{lemma::finallemma} implies consequently that $(\beta^{[m+1]})^{\wedge \dim(X^{[m+1]})}_p \neq 0$. 

In contrast, there exists a non-zero ``vertical'' tangent vector $t \in T_{\IR,x}X_{\pi(x)}$ by assumption. Hence, we obtain a non-zero tangent vector
\begin{equation*}
	((t,\dots,t),0) \in T_{\IR,(x,\dots,x)}X^{[m]}_{\pi(x)}(\IC) \times_{T_{
	\IR,\pi(x)}S} T_{\IR,x}X_{\pi(x)}(\IC) \subset T_{\IR,p}X^{[m+1]}(\IC)
\end{equation*}
that is annihilated by $d\Delta$. This means that $$\ker((d\Delta)_{p}) \cap T_{\IR,p}X^{[m+1]}(\IC) \neq \{0\}$$ and thus $(\Delta^\ast \beta^{[m]})^{\wedge \dim(X^{[m+1]})}_p = 0$, which concludes the proof.
\end{proof}


\section{Proof of Theorem \ref{theorem:equidistribution}}
\label{section:equidistribution}

As in the statement of the theorem, let $S$ be a smooth, geometrically irreducible variety over a number field $K$, $\pi: A \rightarrow S$ an abelian scheme, $\iota: A \hookrightarrow \IP^N_K$ an immersion, $X \subset A$ a non-degenerate geometrically irreducible subvariety of dimension $d$ over $K$ such that $\pi(X)=S$. We also choose an arbitrary immersion $\kappa: S \hookrightarrow \IP^{M}_K$ of the base. We set $h(x) = h_{\overline{\caO}(1)}(\iota(x))$ for each closed point $x \in A$ and similarly $h(s) = h_{\overline{\caO}(1)}(\kappa(s))$ for each closed point $s \in S$. \textit{In this section, all of the constants depend implicitly on the data introduced so far. We exclusively note further dependencies.} Finally, we let $(x_i) \in X^{\IN}$ be an $X$-generic sequence such that $\hhat(x_i) \rightarrow 0$ as $i \rightarrow \infty$.

\subsection{Convergence of heights 
}

We first note a comparison between the asymptotic ``height'' $\hhat$ and the ordinary projective height $h$. This comparison is in fact well-known and could be derived from the explicit estimates in \cite{Zarhin1972}, whose proof involves Mumford's algebraic theta functions \cite{Mumford1966a,Mumford1967}. A general result, applicable beyond families of abelian varieties, was provided by Silverman \cite{Silverman1987}. We state our own version here, which does not require that $\iota^\ast \caO(1)|_A$ is symmetric.
\begin{lemma} 
	\label{lemma::gaohabegger2}	
	There exist constants $c_{14},c_{15}>0$ such that
	\begin{equation*}
	|\hhat(x)- h(x)| \leq c_{14} \cdot \max \{ 1, h(\pi(x)) \} + c_{15} \cdot \hhat(x)^{1/2} \cdot  \max \{ 1, h(\pi(x)) \}^{1/2}
	\end{equation*}
	for all $x \in A(\IQbar)$.
\end{lemma}


\begin{proof}
	By induction on $\dim(S)$, it suffices to prove the estimate for all $x \in \pi^{-1}(U)$ where $U \subseteq S$ is a Zariski-open. We can hence assume that $\iota^\ast \caO(1)|_A$ is the line bundle associated with a relative effective (Cartier) divisor $D \subset A$ as defined in \cite[Tag 056P]{stacksProjectAuthors2015}. We consider the decomposition  $2 \cdot D = D_{\mathrm{sym}} + D_{\mathrm{anti}}$ into a symmetric divisor $D_{\mathrm{sym}} = D + [-1]^\ast D$ and an anti-symmetric divisor $D_{\mathrm{anti}} = D - [-1]^\ast D$. Both $D_{\mathrm{sym}}$ and $D_{\mathrm{anti}}$ are relative effective Cartier divisors. As in \cite{Silverman1987}, we consider some associated ordinary Weil height functions $h_{D_{\mathrm{sym}}}$ and $h_{D_{\mathrm{anti}}}$. (These height functions are only unique up to some bounded function on $A(\IQbar)$ and we make an arbitrary choice here.) Furthermore, we let $\hhat_{D_{\mathrm{sym}}}$ and $\hhat_{D_{\mathrm{anti}}}$ be the fiberwise Néron-Tate heights as in \textit{loc.\ cit}. We can assume that $h = h_{D_{\mathrm{sym}}} + h_{D_{\mathrm{anti}}}$. In contrast, $\hhat_\iota = \hhat_{D_{\mathrm{sym}}}$ is forced upon us by our normalizations.
	
	Applying \cite[Corollary 7.4]{Silverman1987} to both $D_{\mathrm{sym}}$ and $D_{\mathrm{anti}}$ separately, we get a constant $c_{16} > 0$ such that
	\begin{equation}
		\label{equation::first_inequality}
		|\hhat_{D_{\mathrm{sym}}}(x) - h_{D_{\mathrm{sym}}}(x)| + |\hhat_{D_{\mathrm{anti}}}(x) - h_{D_{\mathrm{anti}}}(x)| \leq c_{16} \cdot \max \{1, h(\pi(x))\}.
	\end{equation}
	
	To prove the lemma, it hence remains to bound $|\hhat_{D_\mathrm{anti}}(x)|$. As the line bundle $\caO(D_{\mathrm{anti}})$ is algebraically equivalent to the trivial line bundle over the generic point of $S$, it is so over a Zariski-open $U \subseteq S$ as well. As we are in an induction on the dimension of $S$, we may hence assume that $\caO(D_{\mathrm{anti}}) \in \mathrm{Pic}^0(A/S)(S)$. By \cite[Corollary 6.8]{Mumford1994}, we know that $\mathrm{Pic}^0(A/S)$ is represented by the relative dual abelian variety $A^\vee$ so that $\caO(D_{\mathrm{anti}})$ corresponds to a section $\sigma^\vee: S \rightarrow A^\vee$. As in \cite[Section 6.2]{Mumford1994}, the relatively ample line bundle $\caO(1)|_A$ defines a polarization $\Lambda(\caO(1)|_A): A \rightarrow A^\vee$, which is a finite surjective map. Passing once more to a Zariski-dense open $U \subseteq S$ if necessary, we can assume that $\sigma^\vee = \Lambda(\caO(1)|_A) \circ \sigma$ for some section $\sigma: S \rightarrow A$. Write $B(\cdot,\cdot)$ for the bilinear form associated with the fiberwise Néron-Tate height $\hhat$. By \cite[Proposition 9.3.6 and Corollary 9.3.7]{Bombieri2006} and the Cauchy-Schwarz inequality, we have
	\begin{align*}
		|\hhat_{D_{\mathrm{anti}}}(x)|
		&= |B(x,  (\sigma \circ \pi)(x))| \\
		&\leq B(x,x)^{1/2} \cdot B((\sigma \circ \pi)(x),(\sigma \circ \pi)(x))^{1/2} \\
		&\leq 4 \cdot \hhat(x)^{1/2} \cdot \hhat(\sigma \circ \pi (x))^{1/2}.
	\end{align*}
	for all $x \in A(\IQbar)$. As $\kappa^\ast \mathcal{O}(1)$ is ample on $S$, standard height estimates in combination with \eqref{equation::first_inequality} yield a constant $c_{17} > 0$ such that
	\begin{align*}
		\hhat(\sigma(s)) 
		=\hhat_{D_{\mathrm{sym}}}(\sigma(s))
		&\leq h_{D_{\mathrm{sym}}}(\sigma(s)) + c_{16} \cdot \max \{1, h(s)\} \\
		&\leq  c_{17} \cdot \max \{1, h(s)\}
	\end{align*}
	for all $s \in S(\IQbar)$. In summary, we obtain
	\begin{equation*}
		|\hhat_{D_{\mathrm{anti}}}(x)| \leq c_{18} \cdot \hhat(x)^{1/2} \cdot \max \{ 1, h(\pi(x)) \}^{1/2}
	\end{equation*}
	for some constant $c_{18}$, which in combination with \eqref{equation::first_inequality} yields the assertion.	
\end{proof}

We can also recall the following central result from \cite{Dimitrov2021a}.

\begin{theorem}
	\label{lemma::gaohabegger}
	There exists a constant $c_{19} > 0$ such that
	\begin{equation*}
	h(\pi(x_i)) < c_{19} \max \{1, \hhat(x_i)\}
	\end{equation*}	
	for all but finitely many $i \in \IN$.
\end{theorem}

\begin{proof}
	Taking into account that $(x_i)$ is $X$-generic, this is \cite[Theorem B.1]{Dimitrov2021} applied to the symmetric line bundle $\caO(1)|_A \otimes [-1]^\ast \caO(1)|_A$. 
\end{proof}

Our next lemma is a rather straightforward consequence.

\begin{lemma} 
	\label{lemma::smallsequence}	
	For each integer $k \geq 0$, set $x_i^{(k)} = \iota([n^k](x_i))$ and
	\begin{equation*}
	\mathcal{l}_k = n^{-2k} \cdot \limsup_{i \rightarrow \infty} \left(h_{\overline{\caO}(1)}(x_i^{(k)})\right).
	\end{equation*}
	Then, $\mathcal{l}_k \in [0,\infty)$ for each integer $k \geq 0$. Furthermore, $\lim_{k \rightarrow \infty} (\mathcal{l}_k) = 0$.
\end{lemma}

\begin{proof}
	Note that \eqref{equation::nerontateheight} implies $\hhat(x_i)=n^{-2k}\hhat([n^k](x_{i}))$ for all integers $i, k \geq 0$. For sufficiently large integers $i$, we thus have
	\begin{align*}
	\left| \hhat(x_i) -  \frac{h_{\overline{\caO}(1)}(x_i^{(k)})}{n^{2k}} \right|
	&=
	\frac{|\hhat([n^k](x_i))-h([n^k](x_i))|}{n^{2k}} \\
	&\leq \frac{c_{14} \cdot \max \{ 1, h(\pi(x_i)) \} + c_{15} \cdot \hhat(x)^{1/2} \cdot  \max \{ 1, h(\pi(x_i)) \}^{1/2}}{n^{2k}}\\
	&\leq \frac{c_{20} \cdot \max \{1, \hhat(x_i) \}}{n^{2k}},
	\end{align*}
	where we used Lemma \ref{lemma::gaohabegger2} for the first inequality and Theorem \ref{lemma::gaohabegger} for the second one. 
	As $\lim_{i \rightarrow \infty} \hhat(x_i) = 0$ by assumption, the assertion follows immediately.
\end{proof}

Recall that $\overline{X}_k \subset \IP^N_K$ is the Zariski closure of $\iota([n^k](X))$. Let $\Gamma_k \subset A \times A$ be the graph of $[n^k]|_X$, and let $\overline{Y}_k \subset \IP^N_K \times \IP^N_K$ be the Zariski closure of $(\iota \times \iota)(\Gamma_k)$. We notice that projection to the first (resp.\ second) factor induces a surjective, birational map $\psi_1: \overline{Y}_k \rightarrow \overline{X}_0$ (resp.\ a surjective map $\psi_2: \overline{Y}_k \rightarrow \overline{X}_k$). (Note that $\overline{X}_0$ is just the Zariski closure of $\iota(X) \subset \IP^N_K$.) We write $y_i^{(k)}$ for the point $(\iota(x_i),x_i^{(k)}) = (\iota(x_i),\iota([n^k](x_i)))\in \overline{Y}_k$.

\begin{lemma} 
	\label{lemma::estimate_heights}	
	For all integers $k \geq 0$, we have
	\begin{equation*}
	0 \leq h_{\overline{\mathcal{O}}(1)}(\overline{X}_k) \leq n^{2k}\mathcal{l}_k.
	\end{equation*}
	For any integers $k,k_1,k_2 \geq 0$, we have 
	\begin{equation*}
	0 \leq h_{\overline{\mathcal{O}}(k_1,k_2)}(\overline{Y}_k) \leq k_1 \mathcal{l}_0 + k_2 n^{2k}\mathcal{l}_k.
	\end{equation*}
\end{lemma}

\begin{proof}
	The proof of the lemma is a straightforward application of Zhang's successive minima \cite[Theorem 5.2]{Zhang1995}. 
	Each point of $\overline{X}_k$ (resp.\ $\overline{Y}_k$) has a non-negative height with respect to $\overline{\caO}(1)$ (resp.\ $\overline{\caO}(k_1,k_2)$). Hence the heights of $\overline{X}_k$ and $\overline{Y}_k$ are non-negative as well. In addition, the sequence $(x_i^{(k)})$ (resp.\ $(y_i^{(k)})$) is $\overline{X}_k$-generic (resp.\ $\overline{Y}_k$-generic). Using Zhang's inequalities, we deduce hence
	\begin{equation*}
		h_{\overline{\mathcal{O}}(1)}(\overline{X}_k)
		\leq
		\liminf_{i \rightarrow \infty} \left( h_{\overline{\caO}(1)}(x_i^{(k)}) \right)
		\leq
		\limsup_{i \rightarrow \infty} \left( h_{\overline{\caO}(1)}(x_i^{(k)}) \right)
		= n^{2k} \ell_k
	\end{equation*}
	and similarly
	\begin{align*}
		h_{\overline{\mathcal{O}}(k_1,k_2)}(\overline{Y}_k) 
		&\leq
		\limsup_{i \rightarrow \infty} \left( 	h_{\overline{\mathcal{O}}(k_1,k_2)}(y_i^{(k)}) \right) \\
		&\leq
		k_1 \cdot \limsup_{i \rightarrow \infty} \left( h_{\overline{\mathcal{O}}(1)}(\iota(x_i)) \right)
		+
		k_2 \cdot \limsup_{i \rightarrow \infty} \left( 	h_{\overline{\mathcal{O}}(1)}(x_i^{(k)}) \right) \\
		&=
		k_1 \ell_0 + k_2 n^{2k} \ell_k. \qedhere
	\end{align*}
\end{proof}

By Lemma \ref{lemma::smallsequence}, this proves already the first part of Theorem \ref{theorem:equidistribution}, namely that 
\begin{equation*}
\hhat(X)= \lim_{k \rightarrow \infty} \left( \frac{h_{\overline{\caO}(1)}(\overline{X}_k)}{n^{2k}} \right) = 0.
\end{equation*}

\subsection{Reductions}

Before continuing with the equidistribution part of the proof, we make some reductions. 

Furthermore, we can assume that $\iota^\ast \caO(1)|_A$ is symmetric (i.e., $[-1]^\ast \caO(1)|_A \approx \caO(1)|_A$). 
Composing the immersion $$\iota \times (\iota \circ [-1]): A \hookrightarrow \IP^N_K \times \IP^N_K$$ with the Segre embedding, we obtain an immersion $\iota^\prime: A \hookrightarrow \IP^{(N+1)^2-1}_K$. As the line bundle $(\iota^\prime)^\ast \caO(1)$ is isomorphic to $\iota^\ast \caO(1) \otimes [-1]^\ast \iota^\ast \caO(1)$, it is fiberwise symmetric. In addition, we have
\begin{equation*}
	\lim_{k \rightarrow 
		\infty} \left(\frac{h_{\overline{\mathcal{O}}(1)}(\iota^\prime([n^k](x_i)))}{n^{2k}} \right) = \hhat(x_i) + \hhat(-x_i) = 2 \hhat(x_i) \longrightarrow 0, \ i \rightarrow \infty,
\end{equation*}
by \cite[Proposition B.2.4 (b)]{Hindry2000}. Replacing $\iota$ with $\iota^\prime$, we can assume that $\iota^\ast \caO(1)|_A$ is fiberwise symmetric.

Additionally, we can assume that $\iota^\ast \caO(1)$ is fiberwise of type $(\delta_1,\delta_2,\dots,\delta_g)\in \IZ^g$ with $\delta_1 \geq 3$ and $\delta_{i} \mid \delta_{i+1}$ ($1\leq i < g$). In fact, we are allowed to replace $\iota: A \hookrightarrow \IP^N_K$ with $\nu \circ \iota: A \hookrightarrow \IP^{\binom{N+3}{N}-1}_K$ where $\nu: \IP^N_K \hookrightarrow \IP^{\binom{N+3}{N}-1}_K$ is the Veronese embedding of degree $3$. As $(\nu \circ \iota)^\ast \mathcal{O}(1) \approx \iota^\ast \mathcal{O}(1)^{\otimes 3}$, the immersion $\nu \circ \iota$ satisfies this additional assumption.\footnote{We remark that the pullback of the Fubini-Study metric on $\IP^{\binom{N+3}{N}-1}_\IC$ by the Veronese embedding is not the same as the Fubini-Study metric on $\IP^N_\IC$, but that the associated heights $\hhat_{\iota}$ and $\hhat_{\nu \circ \iota}$ coincide on closed points.}

In the sequel, we consider functions $f\in \mathscr{C}^0_c(X_{\IC_\nu}^{\mathrm{an}})$ as in the statement of Theorem \ref{theorem:equidistribution}. We need additionally that $f$ is $\mathrm{Gal}(\IC_\nu/K_\nu)$-invariant, which we can simply assume by enlarging $K$ such that $\IC_\nu = K_\nu$. The assertion of the theorem only becomes stronger in this way. 
Approximating $f$ uniformly by 
smooth functions, we can assume without loss of generality that $f\in \mathscr{C}^\infty_{c}(X_{\IC_\nu}^{\mathrm{an}})$. In fact, Theorem \ref{theorem:equidistribution} for smooth test functions $f\in \mathscr{C}^\infty_c(X_{\IC_\nu}^{\mathrm{an}})$ already implies $\mu_\nu(X^{\mathrm{an}}_{\IC_\nu}) \leq 1$ as the left-hand side in \eqref{equation:equidistribution} is always $\leq 1$. Therefore, the right-hand side in \eqref{equation:equidistribution} is continuous in the test function $f$ with respect to the uniform topology. This allows us to replace general continuous test functions with smooth ones in the course of proving Theorem \ref{theorem:equidistribution}.\footnote{In fact, additional estimates show that $\mu_\nu(X^{\mathrm{an}}_{\IC_\nu})=1$ (compare \cite{Gauthier2021,Yuan2021b}), but $\mu_\nu(X^{\mathrm{an}}_{\IC_\nu})\leq 1$ is enough for us.}
 
\subsection{Equidistribution}
We extend $f$ by zero to a smooth function $f_{!}$ on $\overline{X}_{0,\IC_\nu}^\mathrm{an}$ and set $f_k = f_{!} \circ \psi_{1,\IC_\nu}^{\mathrm{an}} \in \mathscr{C}_c^\infty(\overline{Y}_{k,\IC_\nu}^{\mathrm{an}})$ for each integer $k \geq 1$ where the map $\psi_1 : \overline{Y}_k \rightarrow \overline{X}_0$ is defined as in the paragraph before Lemma \ref{lemma::estimate_heights}. 

In the following,  we write $\overline{\mathcal{Y}}_k$ for the Zariski closure of $\overline{Y}_k$ in $\IP_{\caO_K}^N \times \IP_{\caO_K}^N$ and $\overline{\mathcal{O}}(1,1;k,\lambda)$, $\lambda \in [0,1]$, for the $\mathscr{C}^\infty$-hermitian line bundle
\begin{equation*}
\overline{\mathcal{O}}(1,1)|_{\overline{\mathcal{Y}}_k} + \overline{\mathcal{O}}(\lambda n^{2k} f_k) \in \widehat{\Pic}_{\mathscr{C}^\infty}(\overline{\mathcal{Y}}_k).
\end{equation*}
Let further $\mu_{k}$ be the measure on $\overline{Y}_{k,\IC_\nu}^\mathrm{an}$ given by the restriction of the $(d,d)$-form
\begin{equation*}
\frac{(\pr_1^\ast \omega_{\mathrm{FS}} + \pr_2^\ast \omega_{\mathrm{FS}})^{\wedge d}}{n^{2kd}}
\end{equation*}
where $\pr_i: \IP^N_{\IC_\nu} \times \IP^N_{\IC_\nu} \rightarrow  \IP^N_{\IC_\nu}$, $i \in \{1,2\}$, is the projection to the $i$-th factor.

\begin{lemma}
	\label{lemma::expansion}
	Assume that $\lambda \in [0,1]$. Then, there exists a constant $c_{21}=c_{21}(f)>0$ such that
	\begin{equation*}
	\left| h_{\overline{\mathcal{O}}(1,1;k,\lambda)}({\overline{Y}}_k)
	- h_{\overline{\mathcal{O}}(1,1)}({\overline{Y}}_k) 
	- \frac{\delta_\nu \lambda n^{2k(d+1)}}{[K:\IQ]\deg_{\caO(1,1)}(\overline{Y}_k)} \int_{\overline{Y}_{k,\IC_{\nu}}^{\mathrm{an}}} f_k \mu_{k} \right| \leq c_{21} \cdot |\lambda|^2 n^{2k}.
	\end{equation*}
	for every integer $k \geq 0$.
\end{lemma}

\begin{proof}
	We note first that
	\begin{equation*}
	h_{\overline{\mathcal{O}}(1,1;k,\lambda)}({\overline{Y}_k}) = \frac{\overline{\mathcal{O}}(1,1;k,\lambda)^{d+1}}{[K:\IQ](d+1)\deg_{\mathcal{O}(1,1)}({\overline{Y}}_k)}.
	\end{equation*}
	by definition and that Lemma \ref{lemma::intersectionnumber} (a,b) provides an expansion
	\begin{equation*}
	\overline{\mathcal{O}}(1,1;k,\lambda)^{d+1} = \sum_{j=0}^{d+1} \binom{d+1}{j} \cdot (\overline{\mathcal{O}}(1,1)|_{ \overline{\mathcal{Y}}_k})^{d+1-j}\cdot
	\overline{\mathcal{O}}(\lambda n^{2k} f_k)^j.
	\end{equation*}
	Again by the definition of the height, we have
	\begin{equation*}
	\frac{(\overline{\mathcal{O}}(1,1)|_{ \overline{\mathcal{Y} }_k})^{d+1}}{[K:\IQ](d+1)\deg_{\mathcal{O}(1,1)}({\overline{Y}}_k)} = h_{\overline{\mathcal{O}}(1,1)}({\overline{Y}_k}),
	\end{equation*}
	and from \eqref{equation::intersection_integral} we infer that
	\begin{equation*}
	\frac{(\overline{\mathcal{O}}(1,1)|_{ \overline{\mathcal{Y} }_k})^{d} \cdot \overline{\mathcal{O}}(\lambda n^{2k} f_k)}{[K:\IQ]\deg_{\mathcal{O}(1,1)}({\overline{Y}}_k)}
	= \frac{\delta_\nu \lambda n^{2k(d+1)}}{[K:\IQ]\deg_{\mathcal{O}(1,1)}({\overline{Y}}_k)} \int_{\overline{Y}_{k,\IC_{\nu}}^{\mathrm{an}}} f_k \mu_{k}.
	\end{equation*}
	Since $\deg_{\mathcal{O}(1,1)}(\overline{Y}_k) \gg n^{2kd}$ by Lemma \ref{lemma::degree_estimate} below, the assertion of the lemma boils down to the estimate
	\begin{equation*} \label{equation::upperbound1}
	\left\vert \sum_{j=2}^{d+1} \binom{d+1}{j} (\overline{\mathcal{O}}(1,1)|_{ \overline{\mathcal{Y}}_k})^{d+1-j} \cdot \overline{\mathcal{O}}(\lambda n^{2k} f_k)^j \right\vert \ll_{f} |\lambda|^2n^{2k(d+1)}.
	\end{equation*}
	Using \eqref{equation::intersection_integral} a second time, we obtain
	\begin{multline*}
	(\overline{\mathcal{O}}(1,1)|_{ \overline{\mathcal{Y}}_k})^{d+1-j} \cdot \overline{\mathcal{O}}(\lambda n^{2k} f_k)^j \\
	= 
	\delta_\nu \lambda^{j}n^{2kj}\int_{\overline{Y}^{\mathrm{an}}_{k,\IC_\nu}} f_k (dd^c f_k)^{\wedge (j-1)} \wedge (\pr_1^\ast \omega_{\mathrm{FS}} + \pr_2^\ast \omega_{\mathrm{FS}})^{\wedge (d+1-j)}
	\end{multline*}
	for each $j \in \{2, \dots, d+1\}$. By the substitution rule \eqref{equation::substitution_rule}, the integral above equals
	\begin{equation*}
	\int_{X^{\mathrm{an}}_{\IC_\nu}} f (dd^c f)^{\wedge (j-1)} \wedge (\iota^\ast \omega_{\mathrm{FS}} + (\iota \circ [n^k])^\ast \omega_{\mathrm{FS}})^{\wedge (d+1-j)},
	\end{equation*}
	which is $\ll_{f} n^{2k(d+1-j)}$ by Lemma \ref{lemma:rathertrivialestimate}. As $\lambda \in [0,1]$ by assumption, the lemma follows immediately.
\end{proof}

The next lemma gives us also a good control on the degree of $\overline{Y}_k$.

\begin{lemma}
	\label{lemma::degree_estimate}
	For each integer $k \geq 0$,
\begin{equation*}
n^{2kd} \ll \deg_{\caO(1,1)}(\overline{Y}_k) \leq \deg_{\caO(n^{2k},1)}(\overline{Y}_k) \ll n^{2kd}.
\end{equation*}
\end{lemma}

\begin{proof}
This is just a combination of Lemmas \ref{lemma::lowerdegreebound} and \ref{lemma::upperdegreebound}.
\end{proof}

To formulate the next lemma, we introduce the supremum
\begin{equation*}
	\mathcal{l} := \sup_{k \geq 0} \{ \mathcal{l}_k \}.
\end{equation*}
Note that $\mathcal{l} \in [0,\infty)$ by Lemma \ref{lemma::smallsequence}.

\begin{lemma}
	\label{lemma::volume_yuan} 
For every integer $k \geq 0$, we have
\begin{equation*}
\widehat{\vol}_\chi(\overline{\mathcal{O}}(1,1;k,\lambda)) - \overline{\mathcal{O}}(1,1;k,\lambda)^{d+1} \gg_{f,\mathcal{l}} - |\lambda|^2 n^{2(d+1)k}.
\end{equation*}
\end{lemma}

The proof of this lemma uses Ikoma's version of Yuan’s bigness theorem \cite{Yuan2008}.

\begin{proof}
Let $\sigma = \sigma(f)$ be a constant such that $f(x) + \sigma \geq 0$ for all $x \in X_{\IC_\nu}^{\mathrm{an}}$. We set
\begin{equation*}
\overline{\caO}(1,1;k,\lambda,\sigma) := \overline{\mathcal{O}}(1,1;k,\lambda) + \overline{\mathcal{O}}(\lambda n^{2k} \sigma) = \overline{\mathcal{O}}(1,1)|_{\overline{\mathcal{Y}}_k} + \overline{\mathcal{O}}(\lambda n^{2k} (f_k + \sigma))
\end{equation*}
and note that
\begin{equation*}
\widehat{\vol}_\chi(\overline{\mathcal{O}}(1,1;k,\lambda)) - \overline{\mathcal{O}}(1,1;k,\lambda)^{d+1} = \widehat{\vol}_\chi(\overline{\caO}(1,1;k,\lambda,\sigma)) - \overline{\caO}(1,1;k,\lambda,\sigma)^{d+1}
\end{equation*}
(e.g.\ by \cite[Lemmas 2.2 (d)\footnote{One of the referees has kindly pointed out to us that this lemma cites a result from \cite{Demailly1993}, for which the numbering in the electronic version on the author's webpage is different from the printed version. It is meant to cite Lemma 9.3 in the electronic version, which corresponds to Lemma 10.2 in the printed version. \newline
We also remark that the lemma is stated for semipositive line bundles, but extends because of multilinearity.} and 2.5 (a)]{Kuehne2022} and \eqref{equation::intersection_integral}). It hence suffices to bound the difference on the right-hand side from below.

Since $\omega_{\mathrm{FS}}$ is a strictly positive $(1,1)$-form and $f_!$ has compact support, there exists some rational constant $q = q(f)> 0$ such that the $(1,1)$-form 
\begin{equation*}
q \cdot \omega_{\mathrm{FS}} + dd^c f_!
\end{equation*}
on $\overline{X}_{0,\IC_\nu}^{\mathrm{an}}$ is strictly positive. Consequently, the $(1,1)$-form 
\begin{equation*}
q \cdot \mathrm{pr}_1^\ast \omega_{\mathrm{FS}} + dd^c f_k
\end{equation*}
on $\overline{Y}_{k,\IC_\nu}^{\mathrm{an}}$ is semipositive. 

Applying Ikoma’s version \cite[Theorem 3.5.3 and Remark 3.5.4]{Ikoma2013} of Yuan’s bigness theorem to the decomposition
\begin{equation*}
\left(\overline{\mathcal{O}}(1,1)|_{\overline{\mathcal{Y}}_k} + q \lambda n^{2k}  \cdot \overline{\mathcal{O}}(1,0)|_{\overline{\mathcal{Y}}_k} + \overline{\mathcal{O}}(\lambda n^{2k}(f_k + \sigma))\right) - q \lambda n^{2k} \cdot \overline{\mathcal{O}}(1,0)|_{\overline{\mathcal{Y}}_k},
\end{equation*}
of $\overline{\caO}(1,1;k,\lambda,\sigma)$, we obtain that $\widehat{\vol}_\chi(\overline{\caO}(1,1;k,\lambda,\sigma))$ is bounded from below by
\begin{multline*}
\left(\overline{\mathcal{O}}(1,1)|_{\overline{\mathcal{Y}}_k} + q \lambda n^{2k}  \cdot \overline{\mathcal{O}}(1,0)|_{\overline{\mathcal{Y}}_k} + \overline{\mathcal{O}}(\lambda n^{2k}(f_k+\sigma))\right)^{d+1} \\
- (d+1) \left(\overline{\mathcal{O}}(1,1)|_{\overline{\mathcal{Y}}_k} + q \lambda n^{2k}  \cdot \overline{\mathcal{O}}(1,0)|_{\overline{\mathcal{Y}}_k} + \overline{\mathcal{O}}(\lambda n^{2k}(f_k+\sigma))\right)^{d} \cdot \left(q \lambda n^{2k} \cdot \overline{\mathcal{O}}(1,0 )|_{\overline{\mathcal{Y}}_k}\right).
\end{multline*}
By Lemma \ref{lemma::intersectionnumber} (a), subtracting $\overline{\caO}(1,1;k,\lambda,\sigma)^{d+1}$ from the above expression results in
\begin{multline}
\label{equation::difference}
- \sum_{i=2}^{d+1} \binom{d+1}{i} \left(\overline{\mathcal{O}}(1,1)|_{\overline{\mathcal{Y}}_k} + q \lambda n^{2k}  \cdot \overline{\mathcal{O}}(1,0)|_{\overline{\mathcal{Y}}_k} + \overline{\mathcal{O}}(\lambda n^{2k}(f_k+\sigma))\right)^{d+1-i}\\ \cdot \left( (-q \lambda n^{2k}) \cdot \overline{\mathcal{O}}(1,0)|_{\overline{\mathcal{Y}}_k}\right)^i.
\end{multline}
We claim that the absolute value of \eqref{equation::difference} is $\ll_{f,\ell} |\lambda|^2 n^{2(d+1)k}$. For this purpose, we expand the intersection number
\begin{equation*}
\left(\overline{\mathcal{O}}(1,1)|_{\overline{\mathcal{Y}}_k} + q \lambda n^{2k} \cdot \overline{\mathcal{O}}(1,0)|_{\overline{\mathcal{Y}}_k} + \overline{\mathcal{O}}(\lambda n^{2k} (f_k +\sigma))\right)^{d+1-i} \cdot \left(q \lambda n^{2k} \cdot\overline{\mathcal{O}}(1,0)|_{\overline{\mathcal{Y}}_k}\right)^i
\end{equation*}
and estimate the intersection number
\begin{equation}
\label{equation::general_term}
(\overline{\mathcal{O}}(1,1)|_{\overline{\mathcal{Y}}_k})^{j_1} \cdot (q \lambda n^{2k} \cdot \mathcal{O}(1,0)|_{\overline{\mathcal{Y}}_k})^{j_2} \cdot (\overline{\mathcal{O}}(\lambda n^{2k} (f_k +\sigma)))^{j_3}, \ j_1 + j_2 + j_3 = d+1,
\end{equation}
in the general term of this expansion. Note that only intersection numbers with $j_2 \geq 2$ appear here, so we may assume this freely in the following. If $j_3 \geq 1$, then \eqref{equation::general_term} equals
\begin{equation*}
q^{j_2}(\lambda n^{2k})^{j_2+j_3}\int_{\overline{Y}_{k,\IC_\nu}^{\mathrm{an}}} (f_k + \sigma) (\pr_1^\ast \omega_{\mathrm{FS}}+\pr_2^\ast \omega_{\mathrm{FS}})^{\wedge j_1} \wedge (\pr_1^\ast \omega_{\mathrm{FS}})^{\wedge j_2} \wedge (dd^c f_k)^{\wedge(j_3 -1)}
\end{equation*}
by \eqref{equation::intersection_integral}. In the case $j_3 = 1$, Lemma \ref{lemma::degree_estimate} implies immediately that this is $\ll_{f} n^{2kd}$.
For the remaining case $j_3 \geq 2$, we recall the notation $\alpha_k = (\iota \circ [n^k])^\ast \omega_{\mathrm{FS}}/n^{2k}$ from Section \ref{section:equilibrium}. Using the substitution formula, we can rewrite the above term as
\begin{equation*}
q^{j_2}(\lambda n^{2k})^{j_2+j_3} \int_{X_{\IC_\nu}^{\mathrm{an}}} (f_! + \sigma) (\alpha_0 + n^{2k} \alpha_k)^{\wedge j_1} \wedge \alpha_0^{\wedge j_2} \wedge (dd^ c f_!)^{\wedge (j_3-1)},
\end{equation*}
whose absolute value is $\ll_{f} |\lambda|^2 n^{2(d+1)k}$ by Lemma \ref{lemma:rathertrivialestimate} and our assumption $\lambda \in [0,1]$. In summary, the absolute value of the corresponding terms in \eqref{equation::difference} is $\ll_{f} |\lambda|^{2} n^{2(d+1)k}$. 

It remains to find similar estimates for \eqref{equation::general_term} if $j_3 = 0$. This term becomes then
\begin{equation*}
(\overline{\mathcal{O}}(1,1)|_{\overline{\mathcal{Y}}_k})^{j} \cdot (q \lambda n^{2k} \cdot \overline{\mathcal{O}}(1,0)|_{\overline{\mathcal{Y}}_k})^{(d+1)-j}.
\end{equation*}
It is easy to see that the intersection numbers
\begin{equation}
\label{equation::terms}
(\overline{\mathcal{O}}(1,1)|_{\overline{\mathcal{Y}}_k})^{j} \cdot (\overline{\mathcal{O}}( n^{2k},0)|_{\overline{\mathcal{Y}}_k})^{(d+1)-j}, \ 0 \leq j \leq (d-1),
\end{equation}
are non-negative. In fact, this follows from the recursive formula for the Arakelov height as both $\overline{\mathcal{O}}(1,0)$ and $\overline{\mathcal{O}}(1,1)$ are globally generated by sections having norm $\leq 1$ everywhere (compare the proof of \cite[Lemma 5.3 (i)]{Zhang1995}). It is hence sufficient to find an upper bound. With a similar argument and using Lemma  \ref{lemma::intersectionnumber} (a), each of the intersection numbers in \eqref{equation::terms} can be bounded from above by $(\overline{\mathcal{O}}(n^{2k},1)|_{\overline{\mathcal{Y}}_k})^{(d+1)}$.

In order to bound this arithmetic intersection number, let us note that Lemma \ref{lemma::estimate_heights} implies that
\begin{align*}
h_{\overline{\mathcal{O}}(n^{2k},1)}(\overline{Y}_k) 
&= \frac{(\overline{\mathcal{O}}(n^{2k},1)|_{ \overline{\mathcal{Y} }_k})^{d+1}}{[K:\IQ](d+1)\deg_{\mathcal{O}(n^{2k},1)}({\overline{Y}}_k)} \\
&\leq n^{2k} (\mathcal{l}_0 + \mathcal{l}_k).
\end{align*}
Combining this with the upper degree bound from Lemma \ref{lemma::degree_estimate}, we infer that
\begin{equation*}
(\overline{\mathcal{O}}(n^{2k},1)|_{\overline{\mathcal{Y}}_k})^{(d+1)} \ll n^{2(d+1)k} (\mathcal{l}_0 + \mathcal{l}_k) \leq 2 n^{2(d+1)k}\mathcal{l}.
\end{equation*}
Retracing the above estimates,
 we conclude the proof.
\end{proof}

With these preparations, we can already bring the proof of Theorem \ref{theorem:equidistribution} close to its end.

\begin{lemma} 
	\label{lemma:equidistribution_estimate}	
	We have
	\begin{equation*}
	\limsup_{i \rightarrow \infty }\left\vert
	\frac{1}{\# \mathbf{O}_\nu(x_i)}\sum_{y \in \mathbf{O}_\nu(x_i)}{f(y)} - \frac{n^{2kd}}{\deg_{\caO(1,1)}(\overline{Y}_k)}\int_{\overline{Y}_{k,\IC_\nu}^{\mathrm{an}}} f_k \mu_{k}\right\vert \longrightarrow 0
	\end{equation*}
	as $k \rightarrow \infty$.
\end{lemma}

\begin{proof}
	Choose some $\lambda \in (0,1]$ as well as some $\varepsilon>0$.

	In the following, we write 
	\begin{equation*}
		\overline{\mathcal{O}}(1,1;k,\lambda) = (\mathcal{O}(1,1)|_{\overline{\mathcal{Y}}_k}, \{ \Vert \cdot \Vert_{\mu} \}_{\mu \in \Sigma_\infty(K)}).
	\end{equation*}	
	By Lemma \ref{lemma::minkowski}, there exists some positive integer $N_0$ and a non-zero section $\mathbf{s} \in H^0(\overline{\mathcal{Y}}_k, \mathcal{O}(1,1)^{\otimes N_0})$ such that
	\begin{equation*}
	\frac{\delta_{\nu} \cdot \log \Vert \mathbf{s}(x) \Vert_{\nu}^{1/N_0}}{[K:\IQ]} \leq  - \frac{\volh_\chi(\overline{\mathcal{O}}(1,1;k,\lambda))}{[K:\IQ](d+1)(\mathcal{O}(1,1)|_{\overline{Y}_k})^d} + \varepsilon
	\end{equation*}
	for every point  $x \in ({\overline{Y}_k}\setminus \left\vert\Div(\mathbf{s})\right\vert)_{\IC_{\nu}}^{\mathrm{an}}$ and $$\log \Vert \mathbf{s}(x) \Vert_{\mu} \leq 0$$ for every place $\mu \in \Sigma_\infty(K) \setminus \{ \nu \}$ and all points $x \in ({\overline{Y}_k}\setminus \left\vert\Div(\mathbf{s})\right\vert)_{\IC_{\mu}}^{\mathrm{an}}$. 
	Using Lemmas \ref{lemma::degree_estimate} and \ref{lemma::volume_yuan}, we infer that there exists some constant $c_{22}=c_{22}(f,\mathcal{l})>0$ such that
	\begin{equation*}
	\frac{\delta_{\nu} \cdot \log \Vert \mathbf{s}(x) \Vert_{\nu}^{1/N_0}}{[K:\IQ]} \leq  - h_{\overline{\mathcal{O}}(1,1;k,\lambda)}(\overline{Y}_k) + c_{22} \cdot \lambda^2n^{2k} + \varepsilon
	\end{equation*}
	for every point $x \in ({\overline{Y}_k}\setminus \left\vert\Div(\mathbf{s})\right\vert)_{\IC_{\nu}}^{\mathrm{an}}$. Using Lemmas \ref{lemma::estimate_heights} and \ref{lemma::expansion}, we deduce that
	\begin{equation*}
	\frac{\delta_{\nu} \cdot \log \Vert \mathbf{s}(x) \Vert_{\nu}^{1/N_0}}{[K:\IQ]} \leq
	- \frac{\delta_{\nu}\lambda n^{2k(d+1)}}{[K:\IQ]\deg_{\caO(1,1)}(\overline{Y}_k)} \int_{\overline{Y}_{k,\IC_{\nu}}^{\mathrm{an}}} f_{k} \mu_{k}
	+ c_{23} \cdot \lambda^2n^{2k} + \varepsilon.
	\end{equation*}
	for some constant $c_{23}=c_{23}(f,\mathcal{l})>0$ and all $x \in ({\overline{Y}_k}\setminus \left\vert\Div(\mathbf{s})\right\vert)_{\IC_{\nu}}^{\mathrm{an}}$. Through (\ref{equation::height_section}), we can use this to obtain the lower global bound
	\begin{equation*}
	h_{\overline{\mathcal{O}}(1,1;k,\lambda)}(x) \geq \frac{\lambda n^{2k(d+1)}}{[K:\IQ]\deg_{\caO(1,1)}(\overline{Y}_k)} \int_{\overline{Y}_{k,\IC_{\nu}}^{\mathrm{an}}} f_{k} \mu_{k} - c_{23} \cdot \lambda^2n^{2k} - \varepsilon
	\end{equation*}
	for all closed points $x \in \overline{Y}_k \setminus \left\vert\Div(\mathbf{s})\right\vert$. (We suppress $
	\delta_\nu$ where it is possible because of $\delta_{\nu} \geq 1$.) Expanding the left-hand side of this equation, we thus obtain
	\begin{multline}
	\label{equation::mainestimate5}
	h_{\overline{\caO}(1,1)}(x) + \frac{1}{[K:\IQ]}\left( \frac{\lambda n^{2k}}{\# \mathbf{O}_{\nu}(x) }\sum_{y \in \mathbf{O}_{\nu}(x)} f_{k}(y) - \frac{\lambda n^{2k(d+1)}}{\deg_{\caO(1,1)}(\overline{Y}_k)} \int_{{\overline{Y}}_{k,\IC_{\nu}}^\mathrm{an}} f_{k} \mu_{k} \right)
	\\
	\geq - c_{23} \cdot \lambda^2n^{2k}  - \varepsilon
	\end{multline}
	for all closed points $x \in {\overline{Y}_k} \setminus \left\vert\Div(\mathbf{s})\right\vert$. Since $(x_i) \in X^{\IN}$ is a generic sequence, there exists some integer $i_0$ such that $y_i^{(k)} = (\iota(x_i),\iota([n^k](x_i))) \notin \left\vert\Div(\mathbf{s})\right\vert$ for all $i \geq i_0$. By definition of $\ell_k$,
	\begin{align*}
		n^{-2k} \cdot \limsup_{i \rightarrow \infty} \left(h_{\overline{\caO}(1,1)}(y_i^{(k)})\right)
		&\leq
		n^{-2k}	\cdot \limsup_{i \rightarrow \infty} \left(h_{\overline{\caO}(1)}(\iota(x_i))\right)
		+
		n^{-2k}	\limsup_{i \rightarrow \infty} \left(h_{\overline{\caO}(1)}(x_i^{(k)})\right) \\
		&\leq n^{-2k} \cdot \ell_0 + \ell_k.
	\end{align*}
	After canceling $[K:\IQ]^{-1}\lambda n^{2k}$ in \eqref{equation::mainestimate5}, we obtain thus
	 \begin{multline*}
	\liminf_{i \rightarrow \infty} \left( \frac{1}{\# \mathbf{O}_\nu(x_i)}\sum_{y \in \mathbf{O}_\nu(x_i)} f(y) - \frac{n^{2kd}}{\deg_{\caO(1,1)}(\overline{Y}_k)} \int_{{\overline{Y}}_{k,\IC_{\nu}}^\mathrm{an}} f_{k} \mu_{k} \right)
	\\
	\geq
	-  c_{24}\cdot (\lambda + \lambda^{-1} n^{-2k}\mathcal{l}_0 + \lambda^{-1}\mathcal{l}_k) - \lambda^{-1} [K:\IQ]\varepsilon
	\end{multline*}	
	for some constant $c_{24} = c_{24}(f,\mathcal{l})>0$. Working with $-f$ instead of $f$ in our above reasoning, we obtain similarly
	 \begin{multline*}
	 \limsup_{i \rightarrow \infty} \left( \frac{1}{\# \mathbf{O}_\nu(x_i)}\sum_{y \in \mathbf{O}_\nu(x_i)} f(y) - \frac{n^{2kd}}{\deg_{\caO(1,1)}(\overline{Y}_k)} \int_{{\overline{Y}}_{k,\IC_{\nu}}^\mathrm{an}} f_{k} \mu_{k} \right)
	 \\
	 \leq
	 c_{24}\cdot (\lambda + \lambda^{-1} n^{-2k}\mathcal{l}_0 + \lambda^{-1}\mathcal{l}_k) + \lambda^{-1}[K:\IQ]\varepsilon
	 \end{multline*}
	In summary, we infer that
	\begin{multline}
	\label{equation:finalinequality}
	\limsup_{i\rightarrow \infty}\left\vert \frac{1}{\# \mathbf{O}_\nu(x_i)}\sum_{y \in \mathbf{O}_\nu(x_i)}{f(y)} - \frac{n^{2kd}}{\deg_{\caO(1,1)}(\overline{Y}_k)} \int_{\overline{Y}_{k,\IC_\nu}^{\mathrm{an}}} f_k \mu_{k} \right\vert \\ \leq c_{24}\cdot (\lambda + \lambda^{-1} n^{-2k}\mathcal{l}_0 + \lambda^{-1}\mathcal{l}_k) + \lambda^{-1}[K:\IQ]\varepsilon.
	\end{multline}
	Given $\varepsilon_0>0$, set 
	\begin{equation*}
		\lambda = \min \left\{\frac{\varepsilon_0}{3c_{24}}, 1 \right\}
	\end{equation*}
	and 
	\begin{equation*}
	\varepsilon = \frac{\varepsilon_0\lambda}{3[K:\IQ]} = \frac{ \min \{ \varepsilon_0^2/3c_{24}, \varepsilon_0 \}}{3[K:\IQ]}.	
	\end{equation*}
	By Lemma \ref{lemma::smallsequence}, there also exists an integer $k_0(\varepsilon_0)$ such that 
	\begin{equation*}
	n^{-2k}\mathcal{l}_0 + \mathcal{l}_k < \varepsilon_0\lambda/3c_{24}= \min \{\varepsilon_0^2/9c_{24}^2, \varepsilon_0/3c_{24} \}
	\end{equation*}	
	for all $k \geq k_0(\varepsilon_0)$. The right-hand side in \eqref{equation:finalinequality} is less than $\varepsilon_0$ if $k \geq k_0(\varepsilon_0)$, whence the assertion of the lemma.
\end{proof}

Let us next establish the asymptotics of the integrals appearing in the last lemma.

\begin{lemma} 
	\label{lemma:convergence_of_integrals}
	As $k \rightarrow \infty$, we have
	\begin{equation*}
	\int_{\overline{Y}^{\mathrm{an}}_{k,\IC_\nu}} f_k \mu_k \longrightarrow \int_{X^{\mathrm{an}}_{\IC_\nu}} f \beta^{\wedge d}
	\end{equation*}
	where $\beta$ is the smooth $(1,1)$-form on $X_{\IC_\nu}^{\mathrm{an}}$ introduced in Lemma \ref{lemma:equilibrium}.
\end{lemma}

\begin{proof}
	By the substitution rule, we have 
	\begin{align*}
	\int_{\overline{Y}^{\mathrm{an}}_{k,\IC_\nu}} f_k \mu_k 
	&= n^{-2kd} \int_{X_{\IC_\nu}^{\mathrm{an}}} f \left(\iota^\ast \omega_{\mathrm{FS}} + (\iota \circ [n^k])^\ast \omega_{\mathrm{FS}}\right)^{\wedge d}.
	\end{align*}
	The assertion follows immediately from Lemmas \ref{lemma:equilibrium} and \ref{lemma:rathertrivialestimate}.
\end{proof}

So far, we have only established (in Lemma \ref{lemma::degree_estimate}) and needed the following degree inequalities:
\begin{equation*}
0<\liminf_{k \rightarrow \infty}\left(\frac{\deg_{\caO(1,1)}(\overline{Y}_k)}{n^{2kd}}\right)\leq\limsup_{k \rightarrow \infty}\left(\frac{\deg_{\caO(1,1)}(\overline{Y}_k)}{n^{2kd}}\right)<\infty.
\end{equation*}
Having almost proven equidistribution in Lemma \ref{lemma:equidistribution_estimate} above, we can show that the middle inequality is in fact an equality.\footnote{This would also follow from \cite[Proposition 13]{Gauthier2019}, but it seems noteworthy that equidistribution can be also used to prove this assertion and so we follow this approach. Furthermore, this choice makes our argument more self-contained and less technical.}

\begin{lemma} 
	\label{lemma::proportionality_constant}
	The limit
	\begin{equation}
	\mathcal{k} = \label{equation::degree_limit}
	\lim_{k \rightarrow \infty} \left(\frac{\deg_{\caO(1,1)}(\overline{Y}_k)}{n^{2kd}} \right)
	\end{equation}
	exists in $(0,\infty)$.
\end{lemma}


\begin{proof}
	As $X$ is non-degenerate, Lemma \ref{lemma::finallemma} guarantees that there exists $f\in \mathscr{C}^\infty_c(X_{\IC_\nu}^{\mathrm{an}})$ such that $\int_{X_{\IC_\nu}^{\mathrm{an}}} f \beta^{\wedge d} > 0$. Applying Lemma \ref{lemma:equidistribution_estimate} with this test function, we obtain an integer $k_0(\varepsilon)$ such that
	\begin{equation*}
	\left\vert
	\frac{n^{2k^\prime d}}{\deg_{\caO(1,1)}(\overline{Y}_{k^\prime})} \int_{\overline{Y}_{k^\prime,\IC_\nu}^{\mathrm{an}}} f_{k^\prime} \mu_{k^\prime} 
	- \frac{n^{2kd}}{\deg_{\caO(1,1)}(\overline{Y}_k)} \int_{\overline{Y}_{k,\IC_\nu}^{\mathrm{an}}} f_k \mu_{k}\right\vert \leq \varepsilon
	\end{equation*}
	for all $k, k^\prime \geq k_0(\varepsilon)$. As both integrals converge to the same (strictly) positive limit by Lemma \ref{lemma:convergence_of_integrals}, we infer that
	\begin{equation*}
	\left(\frac{n^{2kd}}{\deg_{\caO(1,1)}(\overline{Y}_k)}\right)_k
	\end{equation*}
	is a Cauchy sequence and hence converges to some real number in $[0,\infty)$. This shows that the limit \eqref{equation::degree_limit} exists at least in $(0,\infty]$. However, the convergence to $\infty$ is precluded by Lemma \ref{lemma::degree_estimate}.
\end{proof}

There is not much left to complete the proof of Theorem \ref{theorem:equidistribution}: Setting $\mu_\nu = \mathcal{k}^{-1}\cdot \beta^{\wedge d}$, we have
\begin{equation*}
\frac{n^{2kd}}{\deg_{\caO(1,1)}(\overline{Y}_k)} \int_{\overline{Y}^{\mathrm{an}}_{k,\IC_\nu}} f_k \mu_k \longrightarrow
\int_{X^{\mathrm{an}}_{\IC_\nu}} f \mu_\nu
\end{equation*}
as $k \rightarrow \infty$. The theorem is thus an immediate consequence of Lemma \ref{lemma:equidistribution_estimate}.

\section{Proof of Theorems \ref{theorem::uniform_mm} and \ref{theorem::bogo}}
\label{section::uniformity}

Except for Proposition \ref{proposition:bogo2} and its proof, $\pi: A \rightarrow S$ denotes a principally polarized abelian scheme of relative dimension $g$ over a quasi-projective $\IQbar$-variety $S$. We also fix an immersion $\kappa: S \hookrightarrow \IP^N_{\IQbar}$. (In contrast to the other parts of this article, we do not assume a priori that $S$ is smooth in this section, as we can easily reduce to this case in the proofs below.) \textit{In this section, all of the constants can depend on the data introduced so far without further mention. We exclusively keep track of additional dependencies.}

We let $\mathcal{c}: S \rightarrow \mathcal{A}_{g,1}$ be the classifying map such that
\begin{equation}
\label{equation::classifying_map}
\begin{tikzcd}
A \ar[r, "\mathcal{c}_\mathcal{B}"] \ar[d, "\pi"] & \mathcal{B}_{g,1} \ar[d, "\pi_{g,1}"] \\
S \ar[r, "\mathcal{c}"] & \mathcal{A}_{g,1}
\end{tikzcd}
\end{equation}
is cartesian. Furthermore, we consider the base change $\mathcal{c}_\mathcal{B}^\prime$ such that the square
\begin{equation*}
\begin{tikzcd}
	A \times_{\mathcal{B}_{g,1}} \mathcal{B}_{g,3} \ar[r, "\mathcal{c}_\mathcal{B}^\prime"] \ar[d] & \mathcal{B}_{g,3} \ar[d] \\
	A \ar[r, "\mathcal{c}_\mathcal{B}"] & \mathcal{B}_{g,1}
\end{tikzcd}
\end{equation*}
is cartesian; here, $\mathcal{B}_{g,3} \rightarrow \mathcal{B}_{g,1}$ is the forgetful functor. Recall that $\mathcal{B}_{g,3}$ is a quasi-projective variety over $\IQbar$ by \cite[Theorem 7.9]{Mumford1994} and the ``lemma of Serre'' \cite{Serre1962}. It is also smooth by a result of Grothendieck \cite[Theorem 2.4.1]{Oort1971}. 

With the next proposition, we prove a more general result than Theorem \ref{theorem::bogo}, which is additionally amenable to a proof by induction on $\dim(S)$.
\begin{proposition}
	\label{proposition:bogo}
	Let $S$ be an irreducible, quasi-projective variety over $\IQbar$, $\pi: A \rightarrow S$ a principally polarized abelian scheme of relative dimension $g$, $\iota: A \hookrightarrow \IP^N_{\IQbar}$ an immersion, and $C \subseteq A$ an irreducible subvariety such that $\pi(C) = S$. 
	Assume that 
	\begin{enumerate}
	\item[(i)] for every geometric point $s$ of $S$, the fiber $C_s \subset A_s$ is a smooth projective curve of genus $\geq 2$;
	\item[(ii)] for every geometric point $s$ of $S$, the variety $C_s - C_s \subseteq A_s$ is not contained in a proper torsion coset of $A_s$;
	\item[(iii)] the map $\mathcal{c}_\mathcal{B}^\prime$ restricts to a generically finite map on $C \times_{\mathcal{B}_{g,1}} \mathcal{B}_{g,3}$.
	\end{enumerate}
	Then, there exist constants $c_{25} = c_{25}(C,\iota)>0$ and $c_{26}=c_{26}(C,\iota)>0$ such that
	\begin{equation}
	\label{equation::uniform_mm}
	\# \{ x \in C_s(\IQbar) \ | \ \hhat(x) < c_{25} \} \leq c_{26} 
	\end{equation}
	for all $s \in S(\IQbar)$.
\end{proposition}
 
Before giving the proof of this proposition, let us indicate how to deduce the uniform Manin-Mumford and the uniform Bogomolov conjecture from it.

\begin{proof}[Proof of Theorem \ref{theorem::uniform_mm} using Theorem \ref{theorem::bogo}]	
	By a specialization argument due to Masser \cite{Masser1989a} (see \cite[Section 3]{Dimitrov2022}), it is enough to prove Theorem \ref{theorem::uniform_mm} for smooth proper genus $g$ curves $C$ defined over $\IQbar$. In this case, it follows immediately from Theorem \ref{theorem::bogo}, which is deduced below.
\end{proof}
 
\begin{proof}[Proof of Theorem \ref{theorem::bogo} using Proposition \ref{proposition:bogo}]	
	
	We first reduce the assertion to the case where $D = [p]$ for a point $p \in \mathcal{C}_{g,n,s}(\IQbar)$. Assume Theorem \ref{theorem::bogo} is already proven in this special case for constants $c_3^\prime(g,n,\iota)$ and $c_4^\prime(g,n,\iota)$. We claim that Theorem \ref{theorem::bogo} then holds in the general case as well with slightly altered constants
	\begin{equation*}
		c_3(g,n,\iota)= \frac{c_3^\prime(g,n,\iota)}{4} \ \ \text{and} \ \ c_4(g,n,\iota)=c_4^\prime(g,n,\iota).
	\end{equation*}
	To prove this, we may assume that there exists some $p \in \mathcal{C}_{g,n,s}(\IQbar)$ such that
	\begin{equation*}
		\hhat([p]-D) \leq c_3(g,n,\iota);
	\end{equation*}
	for otherwise there is nothing to prove. For every point $q \in \mathcal{C}_{g,n,s}(\IQbar)$ in the set \eqref{equation::my_equation}, we have 
	\begin{equation*}
		\hhat([q] - [p]) \leq 2 \hhat([q] - D) + 2 \hhat(D - [p]) \leq 4 \cdot c_3(g,n,\iota) = c_3^\prime(g,n,\iota)
	\end{equation*}
	by the parallelogram law for the Néron-Tate height \cite[Theorem B.5.1 (c)]{Hindry2000}.
	By assumption, there are at most $c_4^\prime(g,n,\iota)$ such points. This completes the reduction. Hence we may and do assume that $D=[p]$ for some $p \in \mathcal{C}_{g,n,s}(\IQbar)$ in the sequel.

	Let $g \geq 2$ and $n \geq 3$ be fixed integers. Consider the Torelli map (with level structure) $\tau_{g,n}: \mathcal{M}_{g,n} \rightarrow \mathcal{A}_{g,n}$ and the universal family $\pi_{g,n}^\prime: \mathcal{C}_{g,n} \rightarrow \mathcal{M}_{g,n}$. We define the pullback family
	\begin{equation}
	\label{equation::pullback_family}
	A = (\tau_{g,n} \circ \pi_{g,n}^\prime)^\ast \mathcal{B}_{g,n} \longrightarrow S = \mathcal{C}_{g,n}.
	\end{equation}
	Consider the injective $\mathcal{C}_{g,n}$-morphism
	\begin{equation*}
	\varphi_{g,n}: \mathcal{C}_{g,n} \times_{\mathcal{M}_{g,n}} \mathcal{C}_{g,n} \longrightarrow A, \ (p, q) \longmapsto [q] - [p] \in A_{p}.
	\end{equation*}
	As the family $\pi_{g,n}: \mathcal{B}_{g,n} \rightarrow \mathcal{A}_{g,n}$ is projective, so is its pullback \eqref{equation::pullback_family} by \cite[Proposition 5.5.5 (iii)]{EGA2}. By \cite[Proposition 5.3.4 (ii) and Théorème 5.5.3 (i)]{EGA2}, $S$ is quasi-projective over $\IQbar$ as $\pi_{g,n}^\prime$ is projective and $\mathcal{M}_{g,n}$ is quasi-projective. For the same reason, the variety $\mathcal{B}_{g,n}$ is quasi-projective over $\IQbar$ and we can choose an immersion $\iota: \mathcal{B}_{g,n} \hookrightarrow \IP^N_{\IQbar}$. 
	
	Let $C$ be an irreducible component of the subvariety $\varphi_{g,n}(\mathcal{C}_{g,n} \times_{\mathcal{M}_{g,n}} \mathcal{C}_{g,n}) \subseteq A$, which evidently satisfies condition (i) of Proposition \ref{proposition:bogo}. As the Jacobian $A_s$ represents $\mathrm{Pic}^0(C_s)$, it is generated by $C_s - C_s$. Hence (ii) is satisfied as well. Condition (iii) follows from the Torelli theorem \cite[Section VI.3]{Arbarello1985} (see also \cite[Lemma 1.11]{Oort1980}) and the fact that, for each curve $Y$ of genus $\geq 2$, the subvariety $Y - Y \subseteq \mathrm{Jac}(Y)$ has dimension $2$. Applying the proposition to each irreducible component $C$ individually, we obtain Theorem \ref{theorem::bogo}.
\end{proof}

In the sequel, we work with a fixed number field $K$ and a fixed archimedean place $\nu \in \Sigma(K)$. Therefore, we write $Z_\IC$ for the base change $Z \times_K \IC_\nu$ and $Z(\IC)$ for the complex analytic space $Z_{\IC_\nu}^{\mathrm{an}}$. In addition, we write $\mu$ and $\mathbf{O}(x)$ instead of $\mu_\nu$ and $\mathbf{O}_\nu(x)$, respectively. Recall the $(1,1)$-form $\beta$ on $A(\IC)$ from Lemma \ref{lemma:equilibrium}. Furthermore, we use the notations $A^{[n]}$, $X^{[n]}$, $\pi^{[n]}$, and $\iota^{[n]}$ as in Section \ref{section:equilibrium}. We write $\eta$ for the generic point of the base variety $S$.

Instead of proving Proposition \ref{proposition:bogo} directly, we first reduce to another proposition that is easier to work with.
 
\begin{proposition}
	\label{proposition:bogo2}
	Let $K$ be a number field and $\nu$ an archimedean place of $K$. Consider a smooth, geometrically irreducible, quasi-projective variety $S$ over $K$, a principally polarized abelian scheme $\pi: A \rightarrow S$ of relative dimension $g$, an immersion $\iota: A \hookrightarrow \IP^N_K$, and a geometrically irreducible subvariety $C \subseteq A$ such that $\pi(C) = S$ and $\dim(C)=\dim(S)+1$. 
	
	Assume that the generic stabilizer $\mathrm{Stab}_{A_\eta}(C_\eta)$ is trivial and that $C^{[m^\prime]} \subseteq A^{[m^\prime]}$ is a non-degenerate geometrically irreducible subvariety for some integer $m^\prime \geq 2$. Then, there exists an algebraic subvariety $Z \subsetneq S$ of codimension $\geq 1$ and constants $c_{27} = c_{27}(C,\iota,m^\prime)$, $c_{28}=c_{28}(C,\iota,m^\prime)>0$ such that
	\begin{equation}
		\label{equation::uniform_mm}
		\# \{ x \in C_s(\IQbar) \ | \ \hhat(x) < c_{27} \} \leq c_{28} 
	\end{equation}
	for all $s \in S \setminus Z (\IQbar)$.
\end{proposition}

Let us start with the reduction.

\begin{proof}[Proof of Proposition \ref{proposition:bogo} using Proposition \ref{proposition:bogo2}]
	
	Using an induction on $\mathcal{s}=\dim(S)$, it clearly suffices to prove \eqref{equation::uniform_mm} for all $x \in U(\IQbar)$ in a dense open subset $U \subseteq S$. In particular, we can assume that $S$ is smooth in the following. Furthermore, we may and do assume that $A$, $S$, $\pi$, $\iota$, and $\kappa$ are all defined over a number field $K$. 
	The case $\dim(S)=0$ reduces to the classical Bogomolov conjecture \cite{Ullmo1998, Zhang1998}, but our proof actually contains this case as well, by specializing to the argument of \cite[Section 4]{Zhang1998} with some unnecessary modifications.
	
	Choose some $m^\prime \geq \dim(S)$. As the generic fiber $C^{[m^\prime]}_\eta$ is irreducible, we may shrink $S$ to ensure that $C^{[m^\prime]}$ is irreducible. (Recall that the irreducible components of $C^{[m^\prime]}$ meeting the generic fiber $\pi^{-1}(\eta)$ are in a one-to-one with the irreducible components of $C^{[m^\prime]}_\eta$. Compare \cite[(0.2.1.8)]{EGA1}.) By \cite[Theorem 1.3 (i)]{Gao2018a}\footnote{In a previous version of this paper, the stronger result \cite[Theorem 1.3 (ii)]{Gao2018a} was used, which missed a condition that is actually violated in the situation here. The reader is referred to the corrigendum \cite{Gao2021b}.}, the fibered self-product $C^{[m^\prime]}$ is non-degenerate.
	
	Writing $\eta$ for the generic point of $S$, the stabilizer $\mathrm{Stab}(C_\eta)$ of $C_\eta$ in $A_\eta$ is a finite torsion subgroup. In fact, each element of the stabilizer induces a non-trivial automorphism of $C_\eta$. 
	By Hurwitz Theorem \cite[Theorem III.3.9]{Miranda1995}, there are only finitely many such automorphisms since $C_\eta$ has genus $\geq 2$. 
	
	The isogeny $A_\eta \twoheadrightarrow A_\eta/\mathrm{Stab}(C_\eta)$ extends from the generic fiber to some dense open set $U \subseteq S$. After shrinking $S$, we may assume that there exists an abelian $S$-scheme $\pi^\prime: A^\prime \rightarrow S$ and an isogeny $q: A \rightarrow A^\prime$. Furthermore, the image $C^\prime= q(C)$ still has the property that $C^\prime_s$ is a projective, irreducible curve for each geometric point $s$ of $S$, although it is not necessarily smooth anymore. By construction, the stabilizer of $C^\prime_\eta$ in $A_\eta^\prime$ is trivial, and the subvariety
	\begin{equation*}
		(C^\prime)^{[m^\prime]} = (q \times \cdots \times q)(C^{[m^\prime]}) \subseteq (A^\prime)^{[m^\prime]}
	\end{equation*}
	is non-degenerate by Lemma \ref{lemma::nondegenerated_isogenies}. 
	
	After shrinking $S$ further, we can also assume that there exists a dual isogeny $q^\prime: A^\prime \rightarrow A$ such that $q \circ q^\prime = [\deg(q_\eta)]_{A^\prime}$ and $q^\prime \circ q = [\deg(q_\eta)]_A$. The pullback $(q^\prime)^\ast \iota^\ast \caO(1)$ along the finite map $q^\prime$ is relatively ample with respect to $\pi^\prime: A^\prime \rightarrow S$ by \cite[Tags 01VJ and 0892 (b)]{stacksProjectAuthors2015}, so there exist positive integers $a,b$ such that the $b$-th power of the line bundle $(q^\prime)^\ast \iota^\ast \caO(1) \otimes (\pi^\prime)^\ast \kappa^\ast \caO(a)$ is very ample by \cite[Tag 0892 (a)]{stacksProjectAuthors2015}. Write $\iota^\prime: A^\prime \hookrightarrow \IP^{N^\prime}_K$ for an associated projective immersion by means of a basis of its global sections. From the functoriality of the fiberwise Néron-Tate height, it is easy to see that, for every point $x \in A^\prime(\IQbar)$,
	\begin{equation}
		\label{equation::height_comparison}
		\hhat_{\iota^\prime}(q(x)) = b \cdot \hhat(q^\prime(q(x))) = b \cdot \deg(q_\eta)^2 \cdot \hhat(x)
	\end{equation}
	where we set $\hhat_{\iota^\prime}(x) = \lim_{k \rightarrow 
	\infty} (h_{\overline{\mathcal{O}}(1)}(\iota^\prime([n^k](x)))/n^{2k})$. 
	The subvariety $C^\prime \subseteq A^\prime$ satisfies the assumptions of Proposition \ref{proposition:bogo2}. Shrinking $S$ further, we may hence assume that 
	\begin{equation*}
		\# \{ x \in C_s^\prime(\IQbar) \ | \ \hhat_{\iota^\prime}(x) < c_{27} \} \leq c_{28}
	\end{equation*}
	for all $s \in S(\IQbar)$. The assertion of Proposition \ref{proposition:bogo} follows from \eqref{equation::height_comparison}.
	\end{proof}

	\begin{proof}[Proof of Proposition \ref{proposition:bogo2}]
	  Our first goal is to prove that, if the assertion of the proposition is violated, there exists a Zariski-dense sequence $(x_i) \in C^\IN$ of closed points such that $\hhat(x_i) \rightarrow 0$. In this case, there exists a Zariski-dense sequence $(s_k) \in S(\IQbar)^\IN$ such that the sets
	\begin{equation*}
		\Sigma_k := \left\{ x \in C_{s_k}(\IQbar) \ | \ \hhat(x) < (k+1)^{-1} \right\}
	\end{equation*}
	satisfy $\# \Sigma_k \rightarrow \infty$ as $k \rightarrow \infty$. Out of this, we can construct a sequence $(x_i) \in C^\IN$ of closed points whose elements are the closed points corresponding to the rational points $\bigcup_{k=0}^{\infty} \Sigma_k \subseteq C(\IQbar)$.
	
	We next show by induction that such a sequence $(x_i)$ is either Zariski-dense in $C^\IN$ or there is nothing left to prove. In fact, suppose that there exists a Zariski-closed subset $Z \subsetneq C$ containing all elements of the sequence $(x_i)$. Then, $Z$ contains also $\Sigma_k$ for any $k \geq 0$. As $(s_k)$ is Zariski-dense, we have $\pi(Z) = S$. By reasons of dimension, the generic fiber $Z_\eta$ is a finite union of closed points in the generic fiber $C_\eta$. Depriving $S$ of a codimension $\geq 1$ subset, we can hence assume that each geometric fiber $Z_s$ consists of closed points whose number is uniformly bounded. But this implies that $\# \Sigma_k$ is uniformly bounded for infinitely many $k \geq 0$, which is a contradiction. The claim follows.
	
	We continue with describing the setting for our equidistribution argument, following Ullmo \cite{Ullmo1998} and Zhang \cite{Zhang1998}. For given integers $m \geq 2$ and $m^\prime \geq 1$, we define the map	
	\begin{align*}
		\Delta_{0}: \ \ &A^{[m m^\prime]}&  &\longrightarrow& & A^{[(m-1)m^\prime]}, \\ &(\underline{x}_1,\underline{x}_2,\dots,\underline{x}_{m}) & &\longmapsto& &(\underline{x}_1-\underline{x}_2, \underline{x}_2-\underline{x}_3, \dots,\underline{x}_{m-1}-\underline{x}_m),
	\end{align*}
	with each $\underline{x}_i$, $1 \leq i \leq m$, indicating an element of $A^{[m^\prime]}$. In addition, we set
	\begin{equation*}
		\Delta = \Delta_{0} \times_S \id_{A^{[m^\prime]}}: \ A^{[m m^\prime]} \times_S A^{[m^\prime]} \longrightarrow A^{[(m-1)m^\prime]} \times_S A^{[m^\prime]}.
	\end{equation*}

	The triviality of $\Stab(C_\eta)$ implies that $\Stab(C_\eta^{[m^\prime]}) = \Stab(C_\eta)^{m^\prime}$ is trivial as well.	Thus, we can apply \cite[Lemma 4.1]{Abbes1997} for $C^{[m^\prime]}_\eta$, obtaining some integer $m_0 > 0$ such that $\Delta_0|_{C^{[mm^\prime]}_\eta}: C^{[mm^\prime]}_\eta \rightarrow \Delta_0(C^{[mm^\prime]})_\eta$ is birational for all $m \geq m_0$. Consequently, there exists a Zariski-open subset $V \subseteq C^{[(m+1)m^\prime]}$ such that $\Delta|_V: V \rightarrow \Delta(V)$ is an isomorphism onto an open subset $\Delta(V)$ of $\Delta(C^{[(m+1)m^\prime]})$.

	For an arbitrary bijection $\varphi = (\varphi_1,\dots,\varphi_{(m+1)m^\prime}): \IN \rightarrow \IN^{(m+1)m^\prime}$, the sequence
	\begin{equation*}
		y_i = (x_{\varphi_1(i)},x_{\varphi_2(i)},\dots,x_{\varphi_{(m+1)m^\prime}(i)}) \in C^{(m+1)m^\prime}
	\end{equation*}
	is again Zariski-dense and satisfies $\hhat_{\iota^{[(m+1)m^\prime]}}(y_i) \rightarrow 0$. Using \cite[Lemma 4.1]{Zhang1998}, we may even assume that $(y_i)$ is Zariski-generic by passing to a subsequence. The image sequence $(\Delta(y_i))$ is then also Zariski-generic in $\Delta(C^{[(m+1)m^\prime]})$. Furthermore, we can assume that $y_i \notin V$ for all $i \in \IN$. 

	For each integer $n \geq 1$, we write $\beta^{[n]}$ for the $(1,1)$-form defined for $\pi^{[n]}: A^{[n]} \rightarrow S$ in Lemma \ref{lemma:equilibrium}. After shrinking $S$, both $C^{[(m+1)m^\prime]}$ and $\Delta(C^{[(m+1)m^\prime]})$ are non-degenerate by Lemmas \ref{lemma::non_degeneracy} and \ref{lemma::non_degeneracy2}. Therefore, Theorem \ref{theorem:equidistribution} (and its proof) applies to the sequences $(y_i)$ and $(\Delta(y_i))$. For all functions $f \in \mathcal{C}_c^0(C^{[(m+1)m^\prime]}(\IC))$ and $g \in \mathcal{C}_c^0(\Delta(C^{[(m+1)m^\prime]})(\IC))$, we get
	\begin{equation}
		\label{equation::limit1}
		\frac{1}{\# \mathbf{O}(y_i)}\sum_{z \in \mathbf{O}(y_i)} f(z)
		\longrightarrow
		\int_{C^{[(m+1)m^\prime]}(\IC)} f \mu_1
	\end{equation}
	where
\begin{equation*}
	\mu_1=\mathcal{k}^{-1}_{C^{[(m+1)m^\prime]}} \cdot (\beta^{[(m+1)m^\prime]})|_{C^{[(m+1)m^\prime]}}^{\wedge((m+1)m^\prime + \mathcal{s})}
\end{equation*}	
and
\begin{equation}
	\label{equation::limit2}
	\frac{1}{\# \mathbf{O}(\Delta(y_i))}\sum_{z \in \mathbf{O}(\Delta(y_i))} g(z)
	\longrightarrow
	\int_{\Delta(C^{[(m+1)m^\prime]})(\IC)} g \mu_2
\end{equation}
where
\begin{equation*}
	\mu_2=\mathcal{k}^{-1}_{\Delta(C^{[(m+1)m^\prime]})} \cdot  (\beta^{[mm^\prime]})|_{\Delta(C^{[(m+1)m^\prime]})}^{\wedge ((m+1)m^\prime + \mathcal{s})}
\end{equation*}
as $i \rightarrow \infty$. We abuse notation and consider $\mu_1$ and $\mu_2$ as volume forms in the sequel.

Every $f \in \mathscr{C}_c^0(V(\IC))$ can be written as $f = g \circ \Delta$ for some $g \in \mathscr{C}_c^0(\Delta(V)(\IC)))$. Applying \eqref{equation::limit1} and \eqref{equation::limit2} to $f$ and $g$ respectively, we obtain
\begin{equation*}
	\int_{C^{[(m+1)m^\prime]}(\IC)} f \mu_1
	= \int_{\Delta(C^{[(m+1)m^\prime]})(\IC)} g \mu_2 
	= \int_{C^{[(m+1)m^\prime]}(\IC)} f \Delta^\ast \mu_2
\end{equation*}
for all $f \in \mathscr{C}_c^0(V(\IC))$. Note that the substitution rule used here applies indeed also for singular analytic spaces (see the remark in our subsection on \textit{Notations and Conventions} above). We infer that $\mu_1$ and $\Delta^\ast \mu_2$ coincide on $V(\IC)$. As they are real-analytic and $V(\IC)$ is dense in $C^{[(m+1)m^\prime]}(\IC)$, this implies $\mu_1 = \Delta^\ast \mu_2$. This contradicts Lemma \ref{lemma::non_proportional}. Hence, the assertion of the proposition has to hold true.
\end{proof}

\textbf{Acknowledgements:} The author thanks Laura DeMarco, Gabriel Dill, Ziyang Gao, Thomas Gauthier, Philipp Habegger, Myrto Mavraki, Fabien Pazuki, Harry Schmidt, Robert Wilms, Xinyi Yuan for advice, comments, discussions, and encouragement. In particular, he thanks Laura DeMarco for informing him about Gauthier and Vigny's result \cite[Proposition 13]{Gauthier2019}, Xinyi Yuan for sharing the preprint \cite{Yuan2021}, and Thomas Gauthier for sharing his preprint \cite{Gauthier2021}. Furthermore, he thanks Thomas Gauthier and Xinyi Yuan for pointing out an error in the author's previous proof of Lemma \ref{lemma:equidistribution_estimate} and the original version of Theorem \ref{theorem:equidistribution}. Finally, he thanks the three referees for their good advice and many comments, which have substantially contributed to both the mathematical and expository quality of this article.

\textbf{Funding:} The author acknowledges financial support of the Swiss National Science Foundation through an Ambizione Grant in the early stage of this project. The author also received funding from the European Union Horizon 2020 research and innovation programme under the Marie Sklodowska-Curie grant agreement No.\ 101027237.


\bibliographystyle{plain}
\bibliography{../../Bibliography/references}

\end{document}